\documentclass[11pt]{article}
\usepackage[utf8]{inputenc}
\usepackage{lmodern}
\usepackage{amsfonts,amssymb,amsmath,amsthm}
\usepackage{bbm}
\usepackage{graphicx}
\usepackage{graphics}
\usepackage{enumerate}
\usepackage[normalem]{ulem} 
\usepackage{rotating}
\usepackage{array}
\usepackage{subfig}
\usepackage{color}
\usepackage{mathrsfs}
\usepackage{rotating}

\newtheorem{thm}{Theorem}[section]
\newtheorem{lem}[thm]{Lemma}
\newtheorem{cor}[thm]{Corollary}

\newtheorem{defi}[thm]{Definition}
\newtheorem{prop}[thm]{Proposition}
\newtheorem{rem}[thm]{Remark}
\newtheorem{hyp}[thm]{Assumption}

\newcommand{\ben}{\vspace{0mm}\begin{equation}}
\newcommand{\een}{\vspace{0mm}\end{equation}}
\newcommand{\be}{\vspace{0mm}\begin{equation*}}
\newcommand{\ee}{\vspace{0mm}\end{equation*}}
\newcommand{\bea}{\vspace{0mm}\begin{equation*}\begin{aligned}}
\newcommand{\eea}{\vspace{0mm}\end{aligned}\end{equation*}}
\newcommand{\bean}{\vspace{0mm}\begin{equation}\begin{aligned}}
\newcommand{\eean}{\vspace{0mm}\end{aligned}\end{equation}}

\numberwithin{equation}{section}
\numberwithin{figure}{section}

\newcommand{\eps}{\varepsilon}

\newcommand{\Bfl}{\Big\lfloor}
\newcommand{\Bfr}{\Big\rfloor}

\newcommand{\Co}{\mathcal{C}}

\DeclareMathOperator{\Var}{Var}

\newcommand{\supp}{\textnormal{supp}}

\def\D{\mathbb{D}}

\def\N{\mathbb{N}}
\def\P{\mathbb{P}}
\def\R{\mathbb{R}}
\def\C{\mathbb{C}}
\def\E{\mathbb{E}}

\def\ind{{\mathchoice {\rm 1\mskip-4mu l} {\rm 1\mskip-4mu l}
{\rm 1\mskip-4.5mu l} {\rm 1\mskip-5mu l}}}

\title{Renewal in Hawkes processes with self-excitation and inhibition}

\author{Manon Costa\thanks{Institut de Mathématiques de Toulouse, UMR 5219; Université de Toulouse, CNRS, UPS IMT, F-31062 Toulouse Cedex 9, France; E-mail: \texttt{manon.costa@math.univ-toulouse.fr}},
\quad Carl Graham\thanks{CMAP, CNRS, \'Ecole polytechnique, Institut Polytechnique de Paris,
91128 Palaiseau, France. \texttt{carl.graham@polytechnique.edu }},
\quad Laurence Marsalle\thanks{Univ. Lille, CNRS, UMR 8524 - Laboratoire Paul Painlev\'e, F-59000 Lille, France; E-mail: \texttt{laurence.marsalle@math.univ-lille1.fr}},
\quad Viet Chi Tran\thanks{Univ. Lille, CNRS, UMR 8524 - Laboratoire Paul Painlev\'e, F-59000 Lille, France; E-mail: \texttt{chi.tran@math.univ-lille1.fr}}
}

\date{\today}

\begin{document}

\maketitle

\begin{abstract}
This paper investigates
Hawkes processes on the positive real line exhibiting both self-excitation and inhibition. Each
point of this point process  impacts its future intensity by the addition of a signed reproduction function.
The case of a nonnegative reproduction function corresponds to self-excitation,
and has been widely investigated in the literature.
In particular, there exists a cluster representation of the Hawkes process which
allows to apply results known for Galton-Watson trees.
In the present paper, using renewal techniques, we establish limit theorems for Hawkes processes that have reproduction functions which are signed and have bounded support.
We notably prove exponential concentration inequalities, and thus extend results of
Reynaud-Bouret and Roy~\cite{reynaudbouretroy} which were proved
for nonnegative reproduction functions using this cluster representation which is no longer valid
in our case.
An important step for this is to establish the existence of exponential moments for renewal times of M/G/$\infty$ queues that appear naturally in our problem. These results have their own interest, independently of the original problem for the Hawkes processes.

\end{abstract}



\bigskip\noindent
\emph{Acknowledgments}:
The authors thank Patricia Reynaud-Bouret for introducing them to Hawkes processes and for proposing the problem of inhibition in Hawkes processes.
All the authors have been supported by the Chair ``Mod\'elisation Math\'ematique et Biodiversit\'e'' of
Veolia Environnement-Ecole Polytechnique-Museum National d'Histoire Naturelle-Fondation X.
L.M. and V.C.T. acknowledge support from Labex CEMPI (ANR-11-LABX-0007-01).

\section{Introduction and main results}\label{sec:intro}

Hawkes processes have been introduced by Hawkes \cite{hawkes}, and are now widely used in many applications, for example: modeling of earthquake occurrences \cite{hawkesadamopoulos,ogata88},
finance~\cite{bacrydelattrehoffmannmuzySPA, bacrydelattrehoffmannmuzy,bacrymuzy2016},
genetics \cite{reynaudbouretschbath}, neuroscience \cite{chevalliercaceresdoumicreynaudbouret,ditlevsenlocherbach,reynaudbouretrivoirardtuleaumalot}.
Hawkes processes are random point processes on the real line (see~\cite{daleyverejones2003,daleyverejones2008,jacod} for an introduction)
where each atom is associated with a (possibly signed) reproduction measure generating further atoms or adding repulsion.

When the reproduction measure is nonnegative, Hawkes and Oakes \cite{hawkesoakes} have provided a cluster representation of the Hawkes processes
based on immigration of ancestors, each of which is at the head of the branching point process of its offspring.
Exponential concentration inequalities for ergodic theorems and tools for statistical applications have been developed,
\emph{e.g.}, by Reynaud-Bouret and Roy~\cite{reynaudbouretroy}
using a coupling \emph{à la} Berbee \cite{berbee}.

For many applications however, it is important to allow the reproduction measure to be a signed measure.
The positive part of the measure can be interpreted as self-excitation, and its negative part as self-inhibition.
For instance, in neuroscience this can be used to model the existence of a latency period
before the successive activations of a  neuron, see \emph{e.g.} \cite{reynaudbouretrivoirardtuleaumalot}.
Br\'emaud and Massouli\'e~\cite{bremaudmassoulie} have devised efficient
techniques based on Poisson point process thinning  (or embedding) for this framework.
Recent works~\cite{Chen2017,Raad2019} provide interesting contributions
in this perspective, which will be further discussed at the end of Section \ref{sec:intro}.

A large part of the literature on Hawkes processes for neuroscience uses large systems approximations by mean-field
 limits (\emph{e.g.}, \cite{chevallier,delattrefournierhoffmann,delattrefournier,ditlevsenlocherbach})
 or stabilization properties (\emph{e.g.},  \cite{duartelocherbachost} using Erlang kernels). Here, we consider
 a single Hawkes process for which  the reproduction measure is a signed measure and concentrate on extending
 the ergodic theorem and concentration inequalities obtained in \cite{reynaudbouretroy} for a
 nonnegative reproduction measure. Similarly to \cite{reynaudbouretroy},
the reproduction measure is assumed to have bounded support.

A main issue here is that when inhibition is present then the cluster representation of \cite{hawkesoakes}
is no longer valid.
An important tool in our study is a coupling construction of the Hawkes process with signed reproduction measure
and of a Hawkes process with a positive measure.
The former is shown to be a thinning of the latter, for which the cluster representation is valid.

We then define renewal times for these general Hawkes processes.
For this purpose, we introduce an auxiliary strong Markov process with states given by
point processes.
This allows to split the sample paths into
the delay and the cycles, the latter being
i.i.d.\ distributed excursions for which we use limit theorems for i.i.d.\ sequences.

In order to obtain concentration inequalities, a main difficulty is to obtain exponential bounds for the tail distribution of the renewal times. In the case in which the reproduction function is nonnegative, we associate to the Hawkes process
an $M/G/\infty$ queue. To our knowledge, it is the first time that the connection with $M/G/\infty$ queues is made. This allows us to control the length of the excursions of the Hawkes process by using powerful Laplace transform techniques from the queuing theory. These results have independent interest in themselves. We then extend our techniques to Hawkes processes
with signed reproduction functions using the coupling.

We shall explain in Remark~\ref{rmk-phi-coupling} how the coupling method presented in this paper
in a simple framework can be extended to a much broader framework.

\subsection{Definitions and notations}
\label{sec:defs}

\subsubsection*{Measure-theoretic and topological framework}

Throughout this paper, an appropriate filtered probability space $(\Omega,\mathcal{F},(\mathcal{F}_t)_{t\ge0},\P)$ satisfying the usual assumptions is given. All processes will be assumed to be adapted.

Let $\mathcal{N}(\R)$ denote the space of counting measures on the real line $\R=(-\infty, +\infty)$
which are boundedly finite; these are the
 Borel measures with values in $\N_0\cup \{+\infty\}$ (where $\N_0=\{0,1,\ldots\}$) which are finite on any bounded  set. The space $\mathcal{N}(\R)$ is endowed with the weak topology
$\sigma(\mathcal{N}(\R),\Co_{bs}(\R))$ and the corresponding Borel $\sigma$-field, where $\Co_{bs}$ denotes the space of continuous functions with bounded support.

If $N$ is in $\mathcal{N}(\R)$ and $I\subset \R$ is an interval then $N|_I$ denotes the restriction of $N$ to $I$.
Then $N|_I$ belongs to the space $\mathcal{N}(I)$ of boundedly finite counting measures on $I$.
By abuse of notation, a point process on $I$ is often identified with its extension which is null
outside of $I$,
and in particular $N|_I \in \mathcal{N}(I)$ with $\ind_I N \in \mathcal{N}(\R)$.
Accordingly, $\mathcal{N}(I)$ is endowed with the trace topology and $\sigma$-field.

A random point process on $I\subset \R$ will be considered as a random variable taking values in the Polish space $\mathcal{N}(I)$.
We shall also consider random processes with sample paths in the Skorohod space
$\D(\R_+,\mathcal{N}(I))$.

All these spaces are
Polish, see~\cite[Prop.~A2.5.III, Prop.~A2.6.III]{daleyverejones2003}.

\subsubsection*{Hawkes processes}

In this paper we study a random point process on the real line $\R=(-\infty,+\infty)$
specified by a stochastic evolution on the half-line $(0,+\infty)$ and
an initial condition given by a point process on the complementary half-line $(-\infty,0]$.
This is  more general than considering a stationary version of the point process
as was done in early papers~\cite{hawkes,hawkesoakes},
does not require its existence, and can be used to prove the latter as in~\cite{bremaudmassoulie}.
The time origin $0$ can be interpreted as the start of some sort of action with regards to the process
(\emph{e.g.}, observation, or computation of statistical estimators).


In the following definition of a Hawkes process with a signed reproduction measure,
the initial condition $N^0$ is always assumed to be $\mathcal{F}_0$-measurable
and $N^h|_{(0,+\infty)}$ to be adapted to $(\mathcal{F}_t)_{t\ge0}$. We refer to \cite[Sect.~7.2]{daleyverejones2003} for the definition of the conditional intensity measure and denote $x^+=\max(x,0)$ for $x\in\R$.

\begin{defi}
\label{def:Hawkes}
Let $\lambda>0$, a signed measurable function $h: (0,+\infty) \to \R$,
and a boundedly finite point process $N^0$ on $(-\infty,0]$
with law $\mathfrak{m}$ be given.
The point process $N^h$ on $\R$ is a Hawkes process on $(0,+\infty)$ with initial condition $N^0$
and reproduction measure $\mu(dt) \triangleq h(t)\,dt$
if $N^h|_{(-\infty,0]}=N^0$ and the conditional intensity measure of $N^h|_{(0,+\infty)}$
w.r.t.\ $(\mathcal{F}_t)_{t\ge0}$
is absolutely continuous w.r.t.\ the Lebesgue measure and has density
\begin{equation}
\Lambda^h: t\in (0,+\infty) \mapsto
\Lambda^h(t)= \biggl(\lambda+\int_{(-\infty,t)} h(t-u)\,N^h(du)\biggr)^+\,.
\label{def:Lambda}
\end{equation}
\end{defi}
This is a special case of the nonlinear Hawkes process defined in~\cite{bremaudmassoulie},
corresponding  to choosing $x\mapsto (\lambda + x)^+$ as the
function $\phi:\R\to\R_+$ in a conditional intensity of the more general form
$\Lambda^{h,\phi}(t)= \phi\bigl(\int_{(-\infty,t)} h(t-u)\,N^h(du)\bigr)$.
We made this choice in order to streamline the mathematical reasoning and keep formulas reasonably readable.
We shall later detail in Remark~\ref{rmk-phi-coupling} how to extend the results to the more general setting.

Hawkes processes can be defined for reproduction measures $\mu$ which are not absolutely continuous
w.r.t.\ the Lebesgue measure, but we shall consider here this case only. This avoids in particular
the issue of multiplicities of points in $N^h$.
Since $h$ is the density of $\mu$, the support of $h$ is naturally defined as the support of the measure $\mu$:
\begin{equation*}
\supp(h) \triangleq \supp(\mu) \triangleq (0,+\infty) \setminus \bigcup_{G \;\text{open}, \;|\mu|(G) =0}G\,,
\end{equation*}
where $|\mu|(dt)=|h(t)|\,dt$ is the total variation measure of $\mu$. We assume
 w.l.o.g.\ that $h = h\ind_{\supp(h)}$ and define
\[
L(h) \triangleq \sup(\supp(h)) \triangleq \sup\{t >0 , |h(t)|>0\}\in [0,+\infty]\,.
\]

The constant $\lambda$ can be considered as the intensity of a Poisson immigration phenomenon on $(0,+\infty)$.
The function $h$ corresponds to self-excitation and self-repulsion phenomena:
each point of $N^h$ increases, or respectively decreases, the conditional intensity measure wherever the appropriately translated
function $h$ is positive (self-excitation), or respectively negative (self-inhibition).
\par
In the sequel, the notation $\P_\mathfrak{m}$ and $\E_\mathfrak{m}$ is used to
specify that $N^0$ has distribution $\mathfrak{m}$. In the case where
$\mathfrak{m}=\delta_{\nu}$
for some $\nu\in \mathcal{N}((-\infty,0])$, we use the notation $\E_\nu$ and $\P_\nu$.
We often consider the case when $\nu=\emptyset$,
the null measure for which there is no point on $(-\infty,0]$.

In Definition \ref{def:Hawkes}, the density $\Lambda^h$ of the conditional intensity measure of $N^h$
depends on $N^h$ itself, hence existence and uniqueness results are needed.
In Proposition~\ref{prop:couplage}, under the further assumptions that $\|h^+\|_1 <1$ and that
\begin{equation*}
\forall t>0,\quad   \int_0^t \E_{\mathfrak{m}} \Big(\int_{(-\infty,0]}h^+(u-s)\,N^0(ds) \Big)\ du < +\infty\,,
\end{equation*}
we prove that the Hawkes processes can be constructed as the solution of the equation
\begin{equation}
\label{eq:Ng}
\left\{
\begin{aligned}
 &N^{h} = N^0+\int_{(0,+\infty)\times(0,+\infty)} \delta_u \ind_{\{\theta\leq \Lambda^{h}(u)\}}\,Q(du,d\theta)\,,
\\
&\Lambda^h(u) =\biggl(\lambda+\int_{(-\infty,u)} h(u-s)\,N^h(ds)\biggr)^+ \,,
 &&u >0,
\end{aligned}
\right.
\end{equation}
where $Q$ is a $(\mathcal{F}_t)_{t\ge0}$-Poisson point process on $(0,+\infty)\times (0,+\infty)$ with
unit intensity, characterized by the fact that for every $t,h,a>0$, the random variable
$Q((t,t+h]\times (0,a])$ is $\mathcal{F}_{t+h}$-measurable, independent of $\mathcal{F}_t$, and Poisson of parameter $h a$.
Such equations have been introduced and studied
in this context by Br\'emaud and Massouli\'e~\cite{bremaudmassoulie}, see also \cite{bremaudnappotorrisi, massoulie}.

Let us remark that the counting process $(N^h_t)_{t\ge0}$ with sample paths in $\D(\R_+,\N)$ defined by $N^h_t=N^h((0,t])$ satisfies a pure jump time-inhomogeneous stochastic differential equation which is equivalent to the formulation \eqref{eq:Ng}.

If $h$ is a nonnegative function  satisfying $\|h\|_1<1$, then
there exists an alternate existence and uniqueness proof
based on a cluster representation involving subcritical continuous-time Galton-Watson trees, see \cite{hawkesoakes},
which we shall describe and use later.

\subsection{Main Results}

Our goal in this paper is to establish limit theorems for a Hawkes process $N^h$ with general reproduction function $h$. We aim at studying the limiting behavior of the process on a sliding finite time window of length $A$. We therefore introduce a time-shifted version of the Hawkes process.
Using classical notations for point processes, for $t\in\R$ we define
\begin{equation}
\label{def:shift}
S_t : N\in \mathcal{N}(\R)\mapsto  S_tN  \triangleq N(\cdot + t) \in \mathcal{N}(\R)\,.
\end{equation}
Then $S_t N$ is the image measure of $N$ by the shift by $t$ units of time, and
if $a<b$ then
\begin{equation}
\label{def:St}
\begin{aligned}
S_t N((a,b]) &= N((t+a,t+b])\,,
\\
(S_t N)|_{(a,b]} &= S_t(N|_{(t+a,t+b]})=N|_{(t+a,t+b]}(\cdot+t)\,
\end{aligned}
\end{equation}
(with abuse of notation between $N|_{(t+a,t+b]}$ and $\ind_{(t+a,t+b]}N$, etc).

The quantities of interest will be of the form
\begin{equation}
\label{eq:forme}
\frac1T \int_0^T f((S_tN^h)|_{(-A,0]})\,dt = \frac{1}{T}\int_0^T f\big(N^h(\cdot+t)|_{(-A,0]}\big) dt
\end{equation}
in which $T>0$ is a finite time horizon, $A>0$ a finite window length, and $f$
belongs to the set $\mathcal{B}_{lb}(\mathcal{N}((-A,0]))$ of real Borel functions on $\mathcal{N}((-A,0])$
which are locally bounded,  \emph{i.e.}, uniformly bounded on
$\{\nu \in \mathcal{N}((-A,0]) : \nu((-A,0])\le n\}$ for each $n\ge1$.
Such quantities appear commonly in the field of statistical inference of random processes;
by convention, time is labeled so that observation has started by time $-A$.

Using renewal theory, we are able to obtain results without any non-negativity assumption on the reproduction function $h$.
We first establish an ergodic theorem and a central limit theorem for such quantities. We then generalize the
concentration inequalities which were obtained by Reynaud-Bouret and Roy \cite{reynaudbouretroy} under the
assumptions that $h$ is a nonnegative subcritical reproduction law.
This leads us to make the following hypotheses.
Recall that $h = h\ind_{\supp(h)}$ and $L(h) \triangleq \sup(\supp(h)) \triangleq \sup\{t >0 , |h(t)|>0\}$.

\begin{hyp}
\label{hyp:h}
The signed measurable function $h: (0,+\infty) \to \R$ is such that
\begin{equation*}
L(h)  < \infty\,,
\qquad
\|h^+\|_1 \triangleq\int_{(0,+\infty)} h^+(t)\,dt <1\,.
\end{equation*}
The distribution $\mathfrak{m}$ of the initial condition $N^0$ is such that
\begin{equation}\label{eq:hyp-momentN0}
 \E_{\mathfrak{m}} \big(N^0(-L(h),0]\big) < \infty.
\end{equation}
\end{hyp}

We consider only the case of bounded support, \emph{i.e.},  of  $L(h)<\infty$, and focus on treating the difficulties
due to $h$ being signed. The techniques we use exploit this bounded support assumption, which is not very restrictive for the statistical estimation
techniques that we have in mind (\emph{e.g.}, \cite{hansenreynaudbouretrivoirard,lamberttuleaubessaihrivoirardbouretlereschereynaud,reynaudbouretrivoirardtuleaumalot}).
The assumption $\int_{(0,+\infty)} h^+(t)\,dt <1$ will be used to exploit the coupling we will construct between the process with reproduction function
$h$ and a dominating  process with reproduction function $h^+$.
Similar assumptions involving $h^+$ or $|h|$ are often made in the litterature,
see~\cite[Assumption~1]{Chen2017} and~\cite[p.6]{Raad2019},~\emph{e.g.}.

Under these assumptions, we may and will assume that the window $A<\infty$ is such that $A\ge L(h)$.
Then the quantities~\eqref{eq:forme} actually depend only on the restriction
$N^0 |_{(-A,0]}$ of the initial condition $N^0$ to $(-A,0]$.
Thus, in the sequel we identify $\mathfrak{m}$ with its marginal on $\mathcal{N}((-A,0])$ with abuse of notation.
Note that even though~\eqref{eq:hyp-momentN0} does not imply that $\E_\mathfrak{m}\big(N^0((-A,0])\big)<\infty$, our results hold under \eqref{eq:hyp-momentN0}, see Remark~\ref{rque:A-L} below.

The following important results for the Hawkes process $N^h$ are obtained using its regeneration structure, which will be investigated using a Markov process we now introduce.

In Proposition \ref{prop:XMarkov} we prove that if $A\ge L(h)$  then
the process $(X_t)_{t\ge 0}$ defined by
\[
X_t\triangleq (S_t N^h)|_{(-A,0]} \triangleq N^h|_{(t-A,t]}(\cdot+t)
\]
is a strong Markov process which admits an unique invariant law denoted by $\pi_A$,
see  Theorem \ref{thm:exist-uniq-inv-law}  below.

We introduce $\tau$, the first return time to $\emptyset$  (the null point process)
for this Markov process defined by
\begin{equation}
\label{def:tau}
\tau \triangleq \inf\{t>0 : X_{t-}\neq \emptyset, X_{t} =\emptyset\}=\inf\{t>0 : N^h[t-A,t)\neq 0, N^h(t-A,t] =0\}\,.
\end{equation}
The probability measure $\pi_A$ on $\mathcal{N}((-A,0])$ can be classically represented as the intensity of an occupation measure over an excursion: for any nonnegative Borel function $f$,
\begin{equation}\label{def:piA}
\pi_A f  
\triangleq
\frac1{\E_\emptyset(\tau)} \E_\emptyset\biggl(\int_0^{\tau} f((S_t N)|_{(-A,0]})  \,dt\biggr) \in [0,\infty]\,.
\end{equation}

Note that we may then construct a Markov process $X_t$ in equilibrium on $\R_+$ and a time-reversed Markov process in equilibrium on $\R_+$, with identical initial conditions (drawn according to $\pi_A$) and independent transitions, and build from these a Markov process in equilibrium on $\R$. This construction yields a stationary version of $N^h$ on $\R$.

We now state our main results, whose proofs are postponed to Section \ref{sec:proofs}.

\begin{thm}[Ergodic theorems]
\label{thm:point-erg+laws}
Let $N^h$ be a Hawkes process with immigration rate $\lambda >0$, reproduction function $h: (0,+\infty) \to \R$,
and initial condition $N^0$ with law $\mathfrak{m}$, satisfying Assumption~\ref{hyp:h}.
Let $A<\infty$ be such that
$A\ge L(h)$, and $\pi_A$ be the probability measure on $\mathcal{N}((-A,0])$ defined by \eqref{def:piA}.
\begin{enumerate}[\rm a)]
\item
\label{lln}
If $f \in \mathcal{B}_{lb}(\mathcal{N}((-A,0]))$ is nonnegative or $\pi_A$-integrable, then
\begin{equation*}
 \frac1T \int_0^T f((S_tN^h)|_{(-A,0]})\,dt \xrightarrow[T\to\infty]{\P_{\mathfrak{m}}-\textnormal{a.s.}} \pi_A f\,.
\end{equation*}
\item
\label{cv-to-equ}
Convergence to equilibrium for large times holds in the following sense:
\[
 \P_\mathfrak{m}\bigl((S_tN^h)|_{[0,+\infty)}  \in \cdot \bigr)
\xrightarrow[t\to\infty]{\textnormal{total variation}}
\P_{\pi_A}(N^h|_{[0,+\infty)}\in \cdot) \,.
\]
\end{enumerate}
\end{thm}

The following result provides the asymptotics of the fluctuations around the convergence result~\ref{lln}),
and yields asymptotically exact confidence intervals for it. We define the variance \begin{equation}
\label{sig2f}
\sigma^2(f)
\triangleq \frac1{\E_\emptyset(\tau)}
\E_\emptyset\biggl(\biggl(\int_0^{\tau} \big(f((S_tN^h)|_{(-A,0]}) -\pi_Af\big)\,dt\biggr)^2 \biggr) .
\end{equation}

\begin{thm}[Central limit theorem]
\label{thm:clt}
Let $N^h$ be a Hawkes process with immigration rate $\lambda >0$, reproduction function $h: (0,+\infty) \to \R$,
and initial law $\mathfrak{m}$, satisfying Assumption~\ref{hyp:h}.
Let $A<\infty$ be such that $A\ge L(h)$,
the hitting time $\tau$ be given by \eqref{def:tau},
and the probability measure $\pi_A$ on $\mathcal{N}((-A,0])$ be given by \eqref{def:piA}.
If $f \in \mathcal{B}_{lb}(\mathcal{N}((-A,0]))$ is $\pi_A$-integrable and satisfies
{$\sigma^2(f)<\infty$} then
\begin{equation*}
\label{eq:clt}
\sqrt{T} \biggl( \frac1T \int_0^T  f((S_tN^h)|_{(-A,0]}) \,dt - \pi_A f \biggr)
 \xrightarrow[T\to\infty]{\textnormal{in law}} \mathcal{N}(0, \sigma^2(f))\,.
\end{equation*}
\end{thm}

Laws of large numbers and central limit theorems for Hawkes processes, as $T\rightarrow +\infty$, have been much investigated in the case of nonnegative reproduction functions $h$ (\emph{e.g.}, \cite{bacrydelattrehoffmannmuzySPA,jaissonrosenbaum2015,jaissonrosenbaum2016,zhu}).
The convergences in these papers concern the instantaneous values of the counting process of
 the point measure $N^h$, and the proofs usually rely on martingale techniques.
 Here the results concern sliding windows of arbitrary finite length of the point measure $N^h$, and are obtained with the renewal approach that is also developed for establishing non-asymptotic exponential concentration bounds, as explained below.

The first entrance time at $\emptyset$ is defined by
\begin{equation}
\label{def:tau0}
\tau_0 \triangleq \inf\{t{\ge}0 : N^h(t-A,t] =0\}\,.
\end{equation}
Recall that $x^+=\max(x,0)$ and $x^-=\max(-x,0)$ for $x\in\R$, and let
 $(x)_{\pm}^k = (x^{\pm})^k$.

\begin{thm}[Concentration inequalities]
\label{thm:non-asympt-exp-bd}
Let $N^h$ be a Hawkes process with immigration rate $\lambda >0$, reproduction function $h: (0,+\infty) \to \R$,
and initial law $\mathfrak{m}$, satisfying Assumption~\ref{hyp:h}.
Let $A<\infty$ be such that $A\ge L(h)$.
Consider the hitting time $\tau$ given by \eqref{def:tau},
the entrance time $\tau_0$ given by  \eqref{def:tau0},
and the probability measure on $\mathcal{N}((-A,0])$ defined in \eqref{def:piA}.
Consider $f \in \mathcal{B}_{lb}(\mathcal{N}((-A,0]))$ taking its values in a bounded interval $[a,b]$,
and define $\sigma^2(f)$ as in \eqref{sig2f} and
\begin{align*}
& c^\pm(f) \triangleq
\sup_{k\ge3}
\Biggl(
\frac2{k!}
\frac{\E_\emptyset\bigl(\bigl(\int_0^{\tau} (f((S_tN^h)|_{(-A,0]}) -\pi_Af)\,dt\bigr)^k_\pm \bigr)}%
{\E_\emptyset(\tau) \sigma^2(f)}
\Biggr)^{\frac1{k-2}}\,,
\\
& c^\pm(\tau)
\triangleq
\sup_{k\ge3}
\biggl(
\frac2{k!}
\frac{\E_\emptyset\bigl( (\tau-\E_\emptyset(\tau))_\pm^k \bigr)}%
{\Var_\emptyset (\tau)}
\biggr)^{\frac1{k-2}}\,,\\
&
c^+(\tau_0)
\triangleq
\sup_{k\ge3}
\biggl(
\frac2{k!}
\frac{\E_{\mathfrak{m}}\bigl((\tau_0-\E_\mathfrak{m}(\tau_0))_+^k\bigr)}%
{\Var_\mathfrak{m}(\tau_0)}
\biggr)^{\frac1{k-2}}\,.
\end{align*}
Then, for all $\varepsilon>0$, $T>0$ and $u\in [0,1)$,
\begin{align}
\label{eq:conc-ineq}
&\P_\mathfrak{m}\biggl( \biggl|\frac1T \int_0^T  f((S_tN^h)|_{(-A,0]})\,dt  - \pi_A f \biggr| \ge \varepsilon \biggr)
\notag\\
&\quad\le
\exp\left(
-\frac{((1-u)T\eps - |b-a| \E_\emptyset(\tau))^2}
{8 T \sigma^2(f) + 4 c^+(f)((1-u)T\eps - |b-a| \E_\emptyset(\tau)) }
\right)
\notag\\
&\qquad+
\exp\left(
-\frac{((1-u)T)\eps - |b-a| \E_\emptyset(\tau))^2}
{8 T \sigma^2(f) + 4 c^-(f)((1-u)T\eps - |b-a| \E_\emptyset(\tau)) }
\right)
\notag\\
&\qquad+
\exp\left(
-\frac{((1-u)T\eps - |b-a| \E_\emptyset(\tau))^2}
{8T |b-a|^2\frac{\Var_\emptyset (\tau)}{\E_\emptyset(\tau)}+ 4 |b-a| c^+(\tau)((1-u)T\eps - |b-a| \E_\emptyset(\tau)) }
\right)
\notag\\
&\qquad+
\exp\left(
-\frac{((1-u)T\eps - |b-a| \E_\emptyset(\tau))^2}
{8T |b-a|^2\frac{\Var_\emptyset (\tau)}{\E_\emptyset(\tau)}+ 4 |b-a| c^-(\tau)((1-u)T\eps - |b-a| \E_\emptyset(\tau)) }
\right)
\notag\\
&\qquad+
\exp\left(
-\frac{(uT\eps  - 2|b-a| \E_\mathfrak{m}(\tau_0))^2}
{8|b-a|^2\Var_\mathfrak{m} (\tau_0) + 4 |b-a| c^+(\tau_0)(uT\eps  - 2|b-a| \E_\mathfrak{m}(\tau_0)) }
\right).
\end{align}
If $N|_{(-A,0]} =\emptyset$ then the last term of the r.h.s. is null and the upper bound is true with $u=0$
in the other terms.
\end{thm}

In the proof of this theorem, we split the integral from $0$ to $T$ into three parts:
an initial integral from $0$ to $\tau_0$,
a sum of a deterministic number converging to infinity of i.i.d.~integrals over cycles,
and a last integral ending at $T$,
see~\eqref{sum-conc-ineq} below.
The control of the first integral requires to control $\tau_0$, and the control of the last integral requires
to control $\tau_0$ and a similar sum of i.i.d.~random variables.
We control the two sums of i.i.d.~random variables by separating the deviations above and below the mean for precision
and using Bernstein's inequality, which explains the presence of four terms involving $\tau$ in
the r.h.s.~of~\eqref{eq:conc-ineq}. The fifth term is obtained from a control
of $\tau_0$ and depends heavily on the initial condition $\mathfrak{m}$. This explains the introduction of the constant $u$ which can be chosen null when $\tau_0=0$.

We now provide a tractable upper bound, using the fact that the hitting time $\tau$
admits an exponential moment (see Proposition \ref{prop:tau}). For simplicity the process starts at $\emptyset$.

\begin{cor}\label{cor:dev_simple}
Under assumptions and notation of Theorem~\ref{thm:non-asympt-exp-bd}, there exists $\alpha>0$ such that
$\E_\emptyset(e^{\alpha\tau})<\infty$.
Let
\[
v=\frac{2(b-a)^2}{\alpha^2}\Big\lfloor\frac{T}{\E_\emptyset(\tau)}\Big\rfloor \E_\emptyset(e^{\alpha\tau}) e^{\alpha \E_\emptyset(\tau)}\,,
\qquad
c= \frac{|b-a|}{\alpha}\,.
\]
Then for all $T>0$, for all  $\varepsilon>0$,
\begin{align*}
\P_\emptyset\biggl( &\biggl|\frac1T \int_0^T  f((S_tN^h)|_{(-A,0]}) \,dt  - \pi_A f \biggr| \ge \varepsilon \biggr)\le 4 \exp\left(-\frac{\Bigl( T\varepsilon-|b-a|\E_\emptyset(\tau) \Bigr)^2}{4 \left(2v+ c(T\varepsilon-|b-a|\E_\emptyset(\tau)) \right)}  \right)\,,
\end{align*}
or equivalently, for all $1 \ge \eta>0$,
\begin{equation}
P_\emptyset\biggl( \biggl|\frac1T \int_0^T  f((S_tN^h)|_{(-A,0]})\,dt - \pi_A f \biggr| \ge \varepsilon_\eta \biggr)\leq \eta\,,
\end{equation}
where
\[
\varepsilon_\eta=\frac{1}{T}\left(|b-a|\E_\emptyset(\tau)-2c\log\big(\frac{\eta}{4}\big)+\sqrt{4c^2\log^2\big(\frac{\eta}{4}\big)-8 v \log\big(\frac{\eta}{4}\big)}\right)\,.
\]
\end{cor}

\begin{rem}\label{rque:A-L}
All these results hold under \eqref{eq:hyp-momentN0} even if
$\E_\mathfrak{m}(N^0((-A,0]))=+\infty$. Indeed,
\begin{align*}
 \frac{1}{T}\int_0^T f\big(N^h(\cdot+t)|_{(-A,0]}\big) \,dt
&=\frac{1}{T}\int_0^{A-L(h)} f\big(N^h(\cdot+t)|_{(-A,0]}\big) \,dt
\\
&\qquad+\frac{1}{T}\int_{A-L(h)}^T f\big(N^h(.+t)|_{(-A,0]}\big) \,dt\,.
\end{align*}
The first r.h.s.\ term converges $\P_\mathfrak{m}$-a.s.\ to zero, even when multiplied by $\sqrt{T}$.
For the second r.h.s.\  term, we can apply the Markov property at time $A-L(h)$ (which will be justified when proving
that  $(S_.N^h)|_{(-A,0]}$ is a Markov process) and show that $$\E_{(S_{A-L(h)}N^h)|_{(-A,0]}}(N^0((-A,0]))<+\infty.$$
\end{rem}

\begin{rem}
\label{rmk-phi-coupling}
As noted after Definition \ref{def:Hawkes}, the Hawkes process $N^h$ is the special
case for $\phi(x)=(\lambda+x)^+$ of the more general setting
in which a function $\phi:\R\to\R^+$ is given and
the Hawkes process $N^{h,\phi}$ is required to have conditional intensity
\begin{equation}
\label{def:lambda_gen}
\Lambda^{h,\phi}(t)=\phi\left(\int_{(-\infty,t)} h(t-u) N^{h,\phi} (du)\right).
\end{equation}
The results of this article can be extended to this more general setting
under the growth assumption that there exist $\lambda$ and $a$ in $[0,\infty)$ such that
\[
\phi(x)\le \lambda + a x^+\,,\quad x\in\R\,,
\]
 and the stability assumption  that the compactly supported function $h$ satisfies
 \begin{equation}
\label{ext-stab}
 a\int_{(0,+\infty)} h^+(t)\,dt <1\,,
\end{equation}
without any additional regularity or monotonicity assumption on $\phi$.
The main point for this is to construct a thinning coupling $N^{h,\phi}\le N^{h^+}$ similar to the coupling
$N^h\le N^{h^+}$  in Proposition~\ref{prop:couplage}~\ref{prop:couplage-ii}) below, for which
technical details can be found in Appendix \ref{app:extension}.
We chose to present this special case first since it contains all the difficulties and
constitutes the case where the loss of information by coupling is the lowest.
\end{rem}

Two recent works also consider the case of signed reproduction functions. In \cite{Chen2017}, an alternative approach for
analyzing multi-dimensional Hawkes processes with self-inhibition is proposed. The intensity functions are of the form
\[
\lambda_j(t)=\phi_j\left( \mu_j +\sum_{k=1}^p \int_{0}^\infty \omega_{k,j}(u)dN_j(t-u)\right),
\quad
j=1,\dots,p\,,
\]
under a number of assumptions, in particular
that the $\phi_j$ are $\alpha_j$-Lipschitz, the Perron-Frobenius eigenvalue
of the matrix
$(\alpha_j\sum_{k} \int_{0}^\infty |\omega_{k,j}(\Delta)| d\Delta)_{j,k}$ is strictly less than $1$,
and that either the functions $\phi_j$ have a common uniform bound or the signed functions $\omega_{k,j}$
vanish outside a common bounded interval (\cite[Assumption~1, Assumption~4]{Chen2017}).
In order to derive concentration inequalities for the Hawkes processes, they apply the theory of weakly
dependent sequences and therefore develop specific coupling techniques in order to control time dependencies.
The very recent work~\cite{Raad2019} provides renewal time points for rather general one-dimensional
Hawkes processes with self-inhibition, using technical splitting methods requiring specific couplings. \\

Both papers \cite{Chen2017,Raad2019} involve sets of assumptions on the reproduction functions that differ from the ones here.
Note that~\eqref{ext-stab} is the natural stability assumption involving the growth bound at infinity
for the dominating process, and that we do not need in the present paper regularity or monotonicity assumptions on $\phi$.
In contrast~\cite{Chen2017,Raad2019} in the spirit of \cite[Theorem~1]{bremaudmassoulie}
make Lipschitz assumptions on $\phi$ and a stability assumption involving the global Lispschitz
constant of $\phi$ which hence involves its worst local modulus of continuity.
Additionally,~\cite{Chen2017} uses
the equivalent of $|h|$ instead of $h^+$ while~\cite{Raad2019} requires that $\phi$ is nondecreasing.
Moreover, the methods in~\cite{Chen2017,Raad2019} are drastically different from ours,
and require other technical assumptions which we do not need to make.


\section{Hawkes processes}

In this Section, we first provide a constructive solution of Eq.~\eqref{eq:Ng},
which yields a coupling between $N^h$ and $N^{h^+}$ satisfying $N^h\leq N^{h^+}$.
The renewal times on which are based the proofs of our main results are the instants at which
the intensity $\Lambda^h$ has returned and then stayed at $\lambda$ for a duration long enough to be sure that
the dependence on the past has vanished, which allows to write the process
in terms of i.i.d.\ excursions.
The coupling will allow us to control the renewal times for $N^h$ by the renewal times for $N^{h^+}$.

When dealing with $h^+$, we use the well-known cluster representation for a Hawkes process
with nonnegative reproduction function. This representation allows us to interpret the renewal times as times at which an $M/G/\infty$ queue is empty,
and we use this interpretation in order to obtain tail estimates for the interval between these times.

\subsection{Solving the equation for the Hawkes process}

The following result follows from an algorithmic proof which will be given in Appendix~\ref{app:proof}.
The algorithmic construction
can be used for simulations, which are shown in Fig.~\ref{fig:eds}.

\begin{prop}
\label{prop:couplage}
Let $Q$ be a $(\mathcal{F}_t)_{t\ge0}$-Poisson point process on $(0,+\infty)\times (0,+\infty)$ with
unit intensity. 
Consider Equation~\eqref{eq:Ng}, \emph{i.e.},
 \begin{equation*}
\left\{
\begin{aligned}
 &N^{h} = N^0+\int_{(0,+\infty)\times(0,+\infty)} \delta_u \ind_{\{\theta\leq \Lambda^{h}(u)\}}\,Q(du,d\theta)\,,
\\
&\Lambda^h(u) =\biggl(\lambda+\int_{(-\infty,u)} h(u-s)\,N^h(ds)\biggr)^+ \,,
 &&u >0\,,
\end{aligned}
\right.
\end{equation*}
in which $h: (0,+\infty)\to \R$ is a signed measurable reproduction function, $\lambda>0$ an immigration rate,
and $N^0$ an initial condition in $\mathcal{N}((-\infty,0])$ with law $\mathfrak{m}$.
Consider the similar equation for $N^{h^+}$ in which $h$ is replaced by $h^+$.
Assume that
\begin{equation} \label{cond_h}
\|h^+\|_1 <1
\end{equation}
and that the distribution $\mathfrak{m}$ of the initial condition $N^0$ satisfies
\begin{equation} \label{cond_N0}
\forall t>0,\quad   \int_0^t \E_{\mathfrak{m}} \Big(\int_{(-\infty,0]}h^+(u-s)\,N^0(ds) \Big)\ du < +\infty.
\end{equation}
\begin{enumerate}[\rm a)]
\item
\label{prop:couplage-i}
Then there exists a pathwise unique strong solution $N^h$ of Equation~\eqref{eq:Ng},
and this solution is a Hawkes process in the sense of Definition \ref{def:Hawkes}.
\item
\label{prop:couplage-ii}
The same holds for $N^{h^+}$, and moreover $N^h \le N^{h^+}$ a.s. (in the sense of measures).
\end{enumerate}
\end{prop}
The main novelty of this Proposition is the coupling obtained in $\ref{prop:couplage-ii})$. Let us first note that the coupling is very strong since the comparison between $N^h$ and $N^{h^+ }$ holds in the sense of measures: each atom of $N^h$ is an atom of $N^{h^+ }$. Moreover, even though couplings are easily derived for Hawkes processes associated with nonnegative reproduction functions, it is not so when the reproductive functions are signed: if $h$ and $g$ are two signed functions such that $h\le g$, then it is not always possible to couple $N^h$ and $N^g$ in such a way that atoms of $N^h$ are atoms of $N^g$ as well. However,
if $h$ is signed and $g$ is nonnegative, then our construction applies and $N^h\le N^g$,
see Appendix \ref{app:proof} for details.
We present the result above using $h$ and $h^+$ since $h^+$ is the least  positive upper bound of $h$.

\begin{rem}
In order to prove the strong existence and pathwise uniqueness of the solution of Eq.~\eqref{eq:Ng},
we propose a proof based on an algorithmic construction similar to the Poisson embedding of \cite{bremaudmassoulie},
also referred in \cite{daleyverejones2008} as thinning. Since this construction is rather classical,
we postpone the proof in Appendix \ref{app:proof}.
A similar result is also proved in these references using Picard iteration techniques,
with Assumption~\eqref{cond_N0} replaced by the stronger hypothesis that there exists $D_{\mathfrak{m}}>0$
such that
\begin{equation}
\forall t>0,\quad   \E_{\mathfrak{m}} \bigg(\int_{(-\infty,0]}|h(t-s)| \,N^0(ds) \bigg) < D_{\mathfrak{m}}\,.
\end{equation}
When $h$ is nonnegative, the result can be deduced from the cluster representation of the self-exciting Hawkes process,
since $N^h([0,t])$ is upper bounded by the sum of the sizes of a Poisson number of sub-critical Galton-Watson trees,
 see \cite{hawkesoakes,reynaudbouretroy}.
\end{rem}

\begin{rem}
Proposition~\ref{prop:couplage} does not require that $L(h)$ be finite. When $L(h)<\infty$ then
the assumption \eqref{cond_N0} can be rewritten as
\begin{equation}
\int_0^{L(h)} \E_\mathfrak{m} \bigg(\int_{(-L(h),0]}h^+(u-s) \,N^0(ds) \bigg)\ du < +\infty\,.
\label{cond_N0:cas_supph_compact}
\end{equation}
A sufficient condition for \eqref{cond_N0:cas_supph_compact} to hold is that $\E_\mathfrak{m}(N^0(-L(h),0])<+\infty$. Indeed, using the Fubini-Tonelli theorem, the l.h.s.\ of \eqref{cond_N0:cas_supph_compact} can be bounded by
$\|h^+\|_1 \,\E_\mathfrak{m}(N^0(-L(h),0])$. {Therefore, the results of Proposition \ref{prop:couplage} hold under Assumptions \ref{hyp:h}}.
\end{rem}

\begin{figure}[!ht]
\begin{center}
\begin{tabular}{cc}
\includegraphics[height=4.5cm,width=7cm]{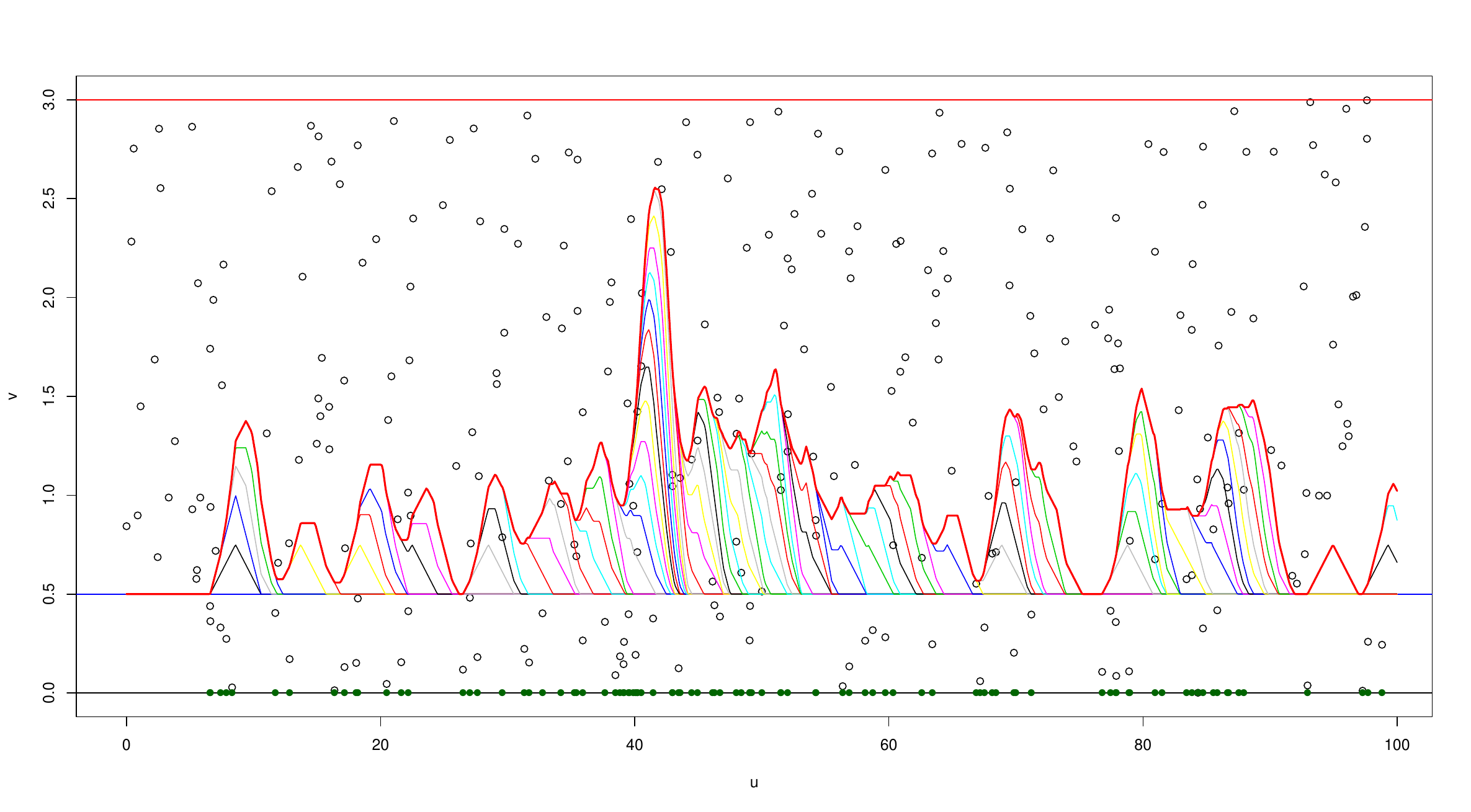} & \includegraphics[height=4.4cm,width=7cm]{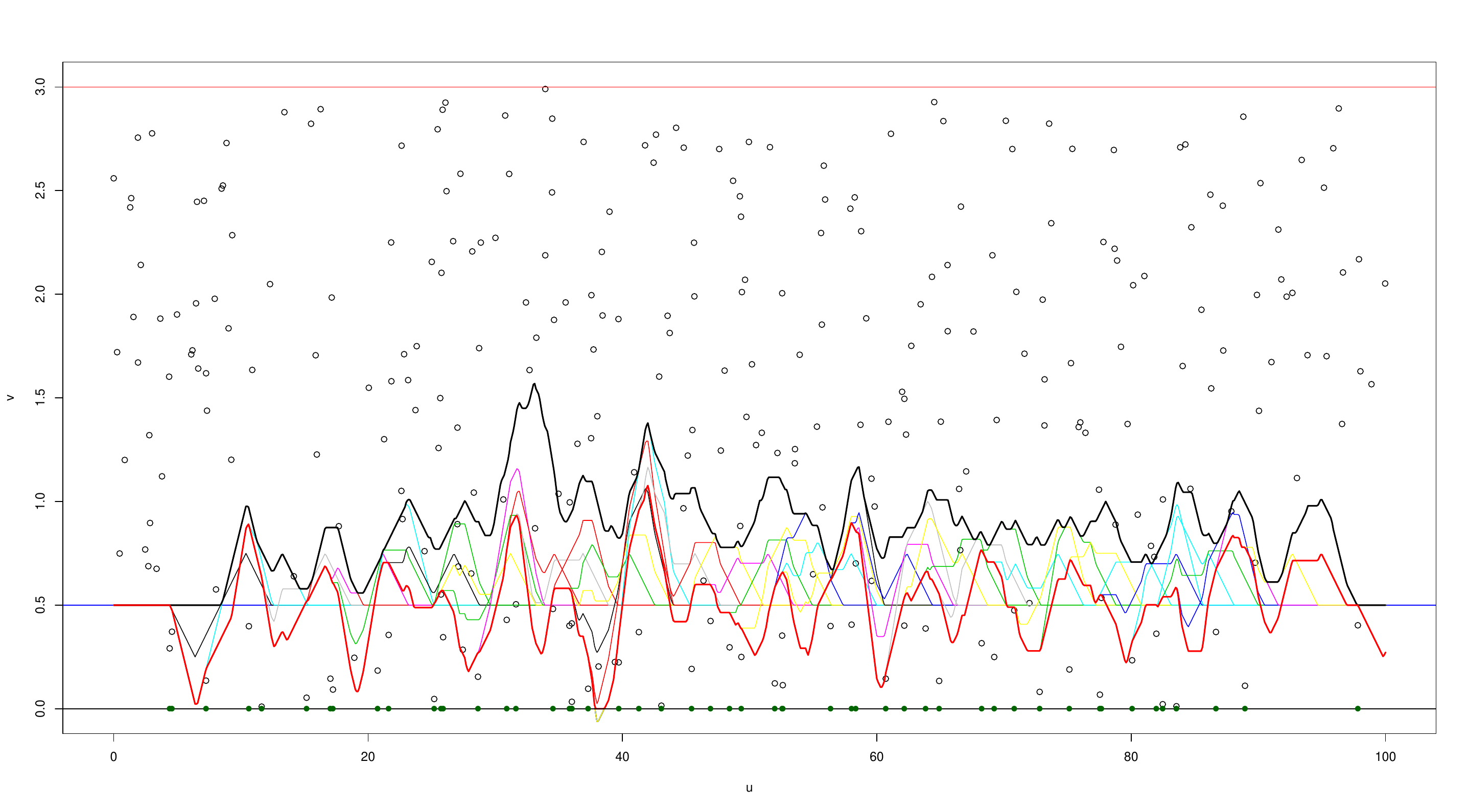}\\
(a) & (b)
\end{tabular}
\caption{
{\small \emph{(a) Hawkes process with a positive reproduction function $h$. (b) Hawkes process with a general reproduction function $h$. The dots in the plane represent the atoms of the Poisson point process $Q$ used for the construction. The atoms of the Hawkes processes are the green dots on the abscissa axis. The bold red curve corresponds to the intensity $\Lambda^h$ and the colored curves represent the partial cumulative contributions of the successive atoms of the Hawkes process. In (b), the bold blue curve  corresponds to the intensity of the dominating Hawkes process with reproduction function $h^+$.}}}\label{fig:eds}
\end{center}
\end{figure}

\subsection{The cluster representation for nonnegative reproduction functions} \label{subsec:cluster}

In this subsection, we consider the case in which  the reproduction function $h$ is nonnegative.
The intensity process of a corresponding Hawkes process can be written, for $t>0$,
\[
\Lambda^h(t) = \lambda + \int_{(-L(h),t)} h(t-u) \,N^h(du)\,.
\]

The first term can be interpreted as an immigration rate of \emph{ancestors}. Let $(V_k)_{k\ge1}$
be the corresponding sequence of arrival times, forming a Poisson  process of intensity $\lambda$.

The second term is the sum of all the contributions of the atoms of $N^h$ before time $t$ and can be seen as self-excitation. If $U$ is an atom of $N^h$, it contributes to the intensity by the addition of the function $t \mapsto h(t-U)$, hence generating new points regarded as its \emph{descendants} or \emph{offspring}. Each individual has a
\emph{lifelength} $L(h)=\sup(\supp(h))$,
 the number of its descendants  follows a Poisson distribution with mean $\|h\|_1$, and the ages at which it gives birth to them have density $h/\|h\|_1$, all this independently. This induces a Galton-Watson process in continuous time,
see \cite{hawkesoakes,reynaudbouretroy}, and Fig. \ref{fig:cluster}.

To each ancestor arrival time $V_k$ we can associate a cluster of times, composed of the times of birth of its descendants.
The condition $\|h\|_1<1$ is a necessary and sufficient condition for the corresponding Galton-Watson process to be sub-critical, which implies that the cluster sizes are finite almost surely. More precisely, if we define $H_k$ by saying that $V_k+H_k$ is the largest time of the cluster associated with $V_k$, then the $(H_k)_{k\ge1}$ are i.i.d random variables independent from the sequence $(V_k)_{k\ge1}$.

Reynaud-Bouret and Roy \cite{reynaudbouretroy} proved the following tail estimate for  $H_1$.
\begin{prop}[{\cite[Prop.~1.2]{reynaudbouretroy}}]
\label{prop:H}
Let us define
\begin{equation}
\label{def:gamma}
\gamma\triangleq\frac{\|h\|_1-\log(\|h\|_1)-1}{L(h)} >0\,.
\end{equation}
Under Assumption~\ref{hyp:h} then
\begin{equation*}
\forall x\ge0,\quad \P(H_1 >x)\le \exp(1-\|h\|_1)\,\exp (-\gamma x)\,,
\end{equation*}
which provides a lower bound for the rate of decay of the cluster length.
\end{prop}

\begin{figure}[!ht]
\begin{center}
\includegraphics[height=5cm,width=8cm]{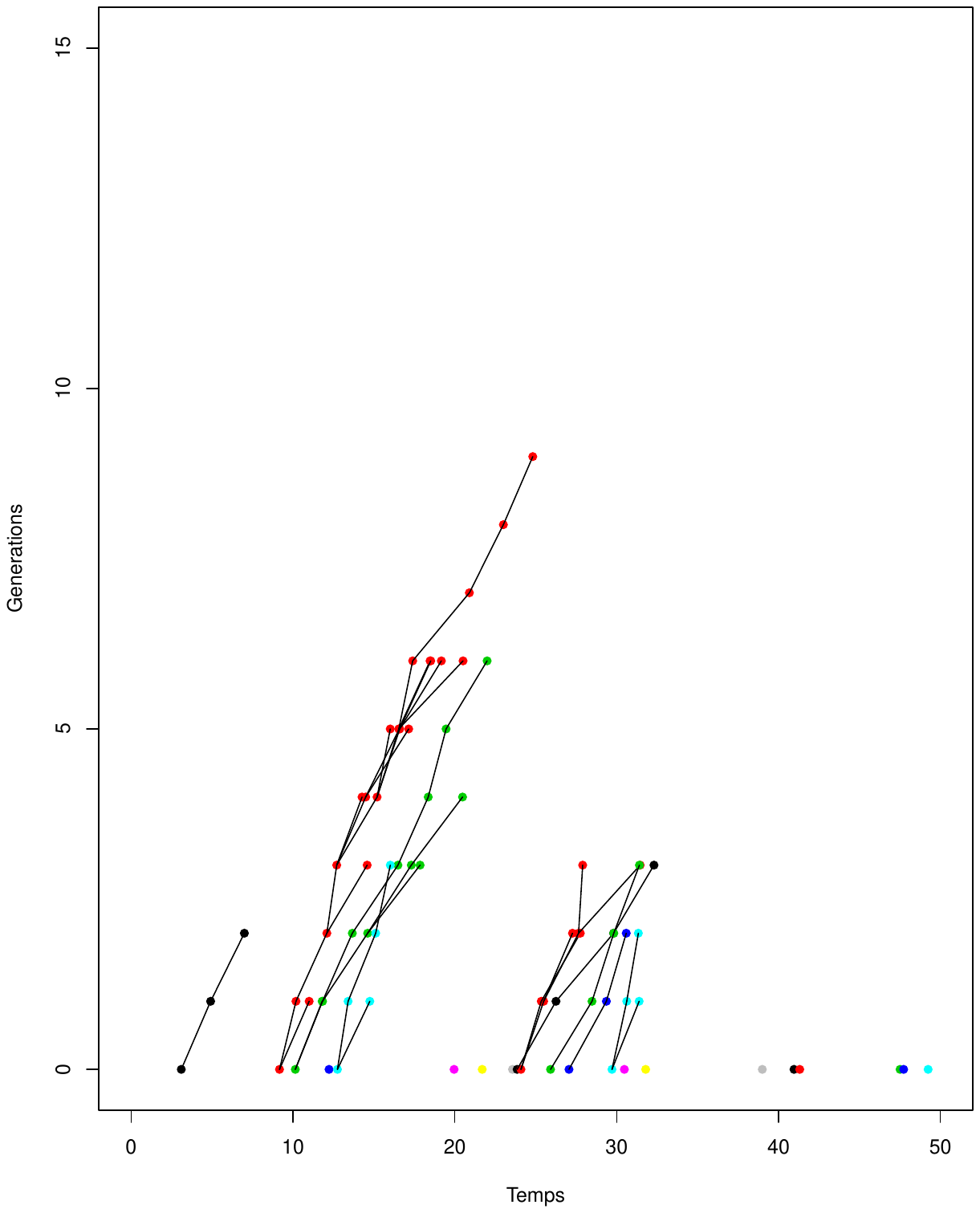}
\caption{{\small \emph{Cluster representation of a Hawkes process with positive reproduction function.
The abscissa of the dots give its atoms. Offspring are colored according to their ancestor,
and their ordinates correspond to their generation in this age-structured Galton-Watson tree.}}}\label{fig:cluster}
\end{center}
\end{figure}

When $h$ is nonnegative, it is possible to associate to the Hawkes process  a $M/G/\infty$ queue.
For  $A \ge L(h)$, we consider that the arrival times of ancestors $(V_k)_{k\ge1}$ correspond to the arrival
of customers in the queue and  associate to the $k$-th customer a service time $\widetilde{H}_k(A) \triangleq H_k+A$.
We assume that the queue is empty at time $0$, and then
the number $Y_t$ of customers in the queue at time $t\ge0$ is given by
\begin{equation}
\label{def:queue}
Y_t=\sum_{k : V_k\le t} \ind_{\{V_k+\widetilde{H}_k(A)>t\}}\,.
\end{equation}
Let $\mathcal{T}_0=0$,  and the successive hitting times of $0$ by the process $(Y_t)_{t\geq 0}$ be given by
\begin{equation} \label{def:tcalk}
\forall k\geq 1,\quad \mathcal{T}_{k}=\inf\{t\geq \mathcal{T}_{k-1},\ Y_{t-}\not=0,\ Y_t=0\}.
\end{equation}
The time interval $[V_1,\mathcal{T}_1)$ is called the first busy period, and is the first time interval during which the queue is never empty. Note that the $\mathcal{T}_{k}$ are times at which the conditional intensity of the underlying Hawkes process has returned to $\lambda$ and there is no remaining influence of its previous atoms,
since  $\widetilde{H}_k(A) \triangleq H_k+A \ge H_k+L(h)$.

Thus the Hawkes process after $\mathcal{T}_{k}$ has the same law as the Hawkes process with initial condition the null point process $\emptyset \in \mathcal{N}((-A,0])$, translated by $\mathcal{T}_{k}$. This allows us to split the random measure $N^h$ into i.i.d.\ parts. We will prove all this in the next section.

We end this part by giving  tail estimates for the $\mathcal{T}_{k}$, which  depend on $\lambda$ and on $\gamma$ given in \eqref{def:gamma}
which respectively control the exponential decays of $\P(V_1>x)$ and $\P(H_1>x)$.

\begin{prop}
\label{prop:decrR_1}
Let Assumption~\ref{hyp:h} hold, and $\gamma$ be given by~\eqref{def:gamma}.
Then
for all $x\ge 0$,
if $\lambda<\gamma$ then
$\P(\mathcal{T}_1>x) = O(\mathrm{e}^{-\lambda x})$
and if $0<\alpha <\gamma\le\lambda$ then
$
\P(\mathcal{T}_1>x)= O(\mathrm{e}^{-\alpha x})
$.
In particular, if $0<\alpha < \min(\lambda, \gamma)$ then $\E(\mathrm{e}^{\alpha\mathcal{T}_1})$ is finite.
\end{prop}

\begin{proof}[Proof of Proposition \ref{prop:decrR_1}]
The proof follows from Proposition \ref{prop:H}, from which we deduce that the service time $\widetilde{H}_1=H_1+A$ satisfies:
\begin{equation}
\label{eq:queueH1tilde}
\P(\widetilde{H}_1>x)=\P(H_1>x-A) \le \exp(-(x-A)\gamma +1-\|h\|_1) = O(\mathrm{e}^{-\gamma x})\,.
\end{equation}
We then conclude by applying Theorem \ref{thm:carl} to the queue $(Y_t)_{t\geq 0}$ defined by \eqref{def:queue}.
\end{proof}

Theorem \ref{thm:carl} in Appendix establishes the decay rates for the tail distributions of $\mathcal{T}_1$ and of the length of the busy period $[V_1,\mathcal{T}_1)$. It has an interest in itself, independently of the results for Hawkes processes considered here.

\section{An auxiliary Markov Process}

When the reproduction function $h$ has a bounded support, $N^h|_{(t,+\infty)}$ depends on $N^h|_{(-\infty,t]}$ only through $N^h|_{(t-L(h),t]}$. The process $t \mapsto N^h|_{(t-L(h),t]}$ will then be seen to be strong Markov, which yields regenerative properties for $N^h$. It is the purpose of this section to formalize that idea by introducing an auxiliary Markov process.

\subsection{Definition of a strong Markov process}
We suppose that Assumption \ref{hyp:h} holds and consider the Hawkes process $N^h$ solution of the corresponding Equation~\eqref{eq:Ng} constructed in Proposition~\ref{prop:couplage}.
We recall that $L(h)<\infty$.
Then, for any $t>0$ and $u\in (-\infty,-L(h)]$, $h(t-u)=0$, and thus
\begin{equation}
\label{def:Lambda-A}
  \Lambda^h(t)  = \biggl(\lambda+\int_{(-\infty,t)}h(t-u)\,N^h(du)\biggr)^+
=  \biggl(\lambda+\int_{(-L(h),t)} h(t-u)\,N^h(du)\biggr)^+\,.
\end{equation}
In particular $N^h|_{(0,+\infty)}$ depends only on the restriction
$N^0|_{(-L(h),0]}$ of the initial condition.

Recall that the shift operator $S_t$ is defined in \eqref{def:shift} and \eqref{def:St}. Note that if $t,s\geq 0$ then $S_{s+t}N^h=S_t S_s N^h=S_s S_tN^h$.
Let $A<\infty$ be such that $A \ge L(h)$.
Consider the $(\mathcal{F}_t)$-adapted process $X=(X_t)_{t\ge 0}$ defined by
\begin{equation}
\label{aux-proc}
X_t = (S_tN^h ) |_{(-A,0]}=N^h|_{(t-A,t]}(\cdot+t) \,,
\end{equation}
\emph{i.e.},
\begin{equation*}
\begin{array}{ccccl}
X_t & : &  \mathcal{B}((-A,0]) & \rightarrow & \R_+\\
 &  & B & \mapsto & X_t(B) = N^h|_{(t-A,t]}(B+t),
 \end{array}
\end{equation*}
 The measure $X_t$  is the point process $N^h$ in the time window $(t-A,t]$, shifted back to $(-A,0]$.
 This is a function of $N^h|_{(-A,+\infty)}$.
Using Equation \eqref{def:Lambda-A} and the remark below it, we see that the law of $N^h|_{(-A,+\infty)}$ depends on the initial condition $N^0$ only through  $N^0|_{(-A,0]}$. Therefore, with abuse of notation, when dealing with the process $(X_t)_{t \ge 0}$ we shall use the notations $\P_\mathfrak{m}$ and $\E_\mathfrak{m}$
even when $\mathfrak{m}$ is a law on $\mathcal{N}((-A,0])$,
and $\P_{\nu}$ and  $\E_{\nu}$ even when $\nu$ is an element of $\mathcal{N}((-A,0])$.

Note that $X$ depends on $A$, and that we omit this in the notation.

\begin{prop}\label{prop:XMarkov}
Let Assumption~\ref{hyp:h} hold.
Let $A<\infty$ be such that $A \ge L(h)$. Then $(X_t)_{t\geq 0}$ defined
in~\eqref{aux-proc} is a strong $(\mathcal{F}_t)_{t\ge0}$-Markov process with initial condition $X_0=N^0|_{(-A,0]}$
and sample paths in the Skorohod space $\D(\R_+,\mathcal{N}((-A,0]))$.
\end{prop}

\begin{proof}
This follows from the fact that $N^h$ is the unique solution of Eq.~\eqref{eq:Ng}.
Indeed, let $T$ be a stopping time. On $\{T<\infty\}$, by definition
\[
X_{T+t} = (S_{T+t}N^h ) |_{(-A,0]} = (S_t S_T N^h ) |_{(-A,0]}\,.
\]
Using that $N^h$ satisfies Eq.~\eqref{eq:Ng} driven by the process $Q$, we have
\begin{align*}
S_T N^h & =  S_T (N^h|_{(-\infty,T]}) +S_T (N^h|_{(T,+\infty)})\\
& =  (S_T N^h)|_{(-\infty,0]} +\int_{(T,+\infty)\times (0,+\infty)} \delta_{u-T} \ind_{\{\theta\leq \Lambda^h(u)\}}Q(du,d\theta)\\
& =  (S_T N^h)|_{(-\infty,0]} +\int_{(0,+\infty)\times (0,+\infty)} \delta_{v}\ind_{\{\theta\leq \widetilde{\Lambda}^h(v)\}} \ S_T Q(dv,d\theta),
\end{align*}
where $S_T Q$ is the (randomly) shifted process with bivariate cumulative distribution function given by
\begin{equation}
\label{shift-Q}
S_TQ((0,t]\times (0,a]) = Q((T,T+t]\times (0,a])\,,
\qquad
t,a>0,
\end{equation}
and where for $v>0$,
\begin{align*}
\widetilde{\Lambda}^h(v)& =  \Lambda^h(v+T)
=  \Big(\lambda+\int_{(-\infty,v)}h(v-s)S_T N^h(ds)\Big)^+.
\end{align*}This shows that $S_T N^h$ satisfies Eq.~\eqref{eq:Ng} driven by $S_TQ$ with initial condition $(S_T N^h)|_{(-\infty,0]}$. Since $A\ge L(h)$,  moreover $S_TN^h|_{(0,+\infty)}$ actually depends only on $(S_T N^h) |_{(-A,0]} \triangleq X_T$.

Let us now condition on $\{T<\infty\}$ and on $\mathcal{F}_T$. Since $Q$ is a $(\mathcal{F}_t)_{t\ge0}$-Poisson point process with unit intensity,
$S_TQ$ is a $(\mathcal{F}_{T+t})_{t\ge0}$-Poisson point process with unit intensity, see Lemma~\ref{lem:Q} for this classic fact.
In particular it is independent of the $\mathcal{F}_T$-measurable random variable $X_T$. Additionally, $X_T$ satisfies Assumption \eqref{cond_N0}, which becomes in this case: for all $r>0$
$$\int_0^r \int_{(-A,0]}h^+(u-s) (S_T N^h)(ds) \ du < +\infty \qquad \P_{\mathfrak{m}}\textnormal{-a.s.}$$
We have indeed that:
\begin{align}
\int_0^r \int_{(-A,0]}&h^+(u-s) (S_TN^h)(ds) du \nonumber \\
 &  = \int_0^r \int_{(-A+T,T]}h^+(T+u-s) N^h(ds) \ du \nonumber \\
 &= \int_T^{T+r} \int_{(-A+T,T]}h^+(v-s) N^h(ds) \ dv \nonumber \\
 &\le  \int_T^{T+r} \int_{(-\infty,0]}h^+(v-s) N^0(ds) \ dv + \int_T^{T+r} \int_{(0,T]}h^+(v-s) N^h(ds) \ dv \nonumber \\
 & \le \int_T^{T+r} \int_{(-\infty,0]}h^+(v-s) N^0(ds) \ dv + \|h^+\|_1 N^h(0,T] \nonumber\\
 &< +\infty\qquad \P_{\mathfrak{m}}\textnormal{-a.s.},\nonumber
\end{align}
since the distribution $\mathfrak{m}$ of $N^0$ satisfies \eqref{cond_N0}, and since we have shown at the end of the proof of Proposition \ref{prop:couplage} that $\E_{\mathfrak{m}}(N^h(0,t]) < +\infty$ for all $t >0$.

Thus the assumptions of Proposition~\ref{prop:couplage} are satisfied, which yields that $(X_{T+t} )_{t\ge0}$ is the pathwise unique, and hence weakly unique, strong solution of Eq.~\eqref{eq:Ng} started at $X_T$ and driven by the
$(\mathcal{F}_{T+t})_{t\ge0}$-Poisson point process $S_TQ$. Hence, it is a process started at $X_T$
which is a $(\mathcal{F}_{T+t})_{t\ge0}$-Markov process with same transition semi-group as $(X_t)_{t\geq 0}$.
If we wish to be more specific, for every bounded Borel function $F$ on $\D(\R_+,\mathcal{N}((-A,0]))$
we set
\[
\Pi F(x) \triangleq \E_x(F((X_t)_{t\geq 0}))
\]
and note that  existence and uniqueness in law for~\eqref{eq:Ng} yield that
\[
\E_x(F((X_t)_{t\geq 0}) \,|\, T<\infty, \mathcal{F}_T) = \Pi F(X_T)\,.
\]
This is the strong Markov property we were set to prove.
\end{proof}
\color{black}

\subsection{Renewal of $X$ at $\emptyset$}

Using $(X_t)_{t\geq 0}$ and Proposition~\ref{prop:XMarkov}, we obtain that if $T$ is a stopping time such that $N^h|_{(T-A,T]}=\emptyset$, then $N^h|_{(T,+\infty)}$ is independent of $N^h|_{(-\infty,T]}$ and behaves as $N^h$ started from $\emptyset$ and translated by $T$. Such renewal times lead to an interesting decomposition of $N^h$, enlightening its dependence structure.

The successive hitting times of $\emptyset \in \mathcal{N}((-A,0])$ for the Markov process $X$ are such renewal times.
This subsection is devoted to the study of their properties.
Recall that we have introduced in \eqref{def:tau} the first hitting time of $\emptyset \in \mathcal{N}((-A,0])$
for $X$, given by
\[
\tau \triangleq \inf\{t>0 : X_{t-}\neq \emptyset, X_{t} =\emptyset\}=\inf\{t>0 : N^h[t-A,t)\neq 0, N^h(t-A,t] =0\}\,.
\]
It depends on $A$, but this  is omitted in the notation. It is natural to study whether $\tau$ is finite or not. When
the reproduction function $h $ is nonnegative, we introduce the queue $(Y_t)_{t \ge 0}$ defined by \eqref{def:queue},
 and its return time to zero $\mathcal{T}_1$ defined in \eqref{def:tcalk}.
 The following result will yield  the finiteness of $\tau$.


\begin{lem}
\label{lem:egalite_temps}
Let Assumption \ref{hyp:h} hold.
Let $A<\infty$ be such that $A \ge L(h)$. Let
$\tau$ and $\mathcal{T}_1$ be defined in \eqref{def:tau} and \eqref{def:tcalk}.
If $h$ is nonnegative then $\P_{\emptyset}(\tau  = \mathcal{T}_1)=1$.
\end{lem}

\begin{proof}
We use the notations defined in Section \ref{subsec:cluster}. To begin with, we remark that $\tau >V_1$.
First, let us consider $t$ such that $V_1<t< \mathcal{T}_1$. By definition, there exists $i\ge1$, such that
$$V_i\le t\le V_i+\widetilde{H}_i(A) =V_i+H_i+A.$$
Since the interval $[V_i, V_i+H_i]$ corresponds to the cluster of descendants of $V_i$, there exists a sequence of points of $N^h$ in $[V_i, V_i+H_i]$ which are distant by less than $L(h)$ and thus less than $A$. Therefore, if $t\in[V_i, V_i+H_i]$, then $N^h(t-A,t]>0$

If $t\in[ V_i+H_i, V_i+H_i+A]$, then $N^h(t-A,t]>0$ as well since $V_i+H_i \in N^h$ (it is the last birth time in the Galton-Watson tree stemming from $V_i$, by definition of $H_i$).
Since this reasoning holds for any $t\le \mathcal{T}_1$, thus $\tau \ge\mathcal{T}_1$.

Conversely, for any $t \in [V_1,\tau)$, by definition of $\tau$ necessarily $N^h(t-A,t]>0$. Thus there exists an atom of $N^h$ in $(t-A,t]$, and from the cluster representation, there exists $i \ge 1$ such that this atom belongs to the cluster of $V_i$, hence to $[V_i,V_i+H_i]$. We easily deduce that
$$V_i\le t\le V_i+H_i+A$$ and thus $Y_t\ge1$, for all $t  \in [V_1,\tau)$. This proves that $\tau\le \mathcal{T}_1$ and concludes the proof.
\end{proof}

To extend the result of finiteness of $\tau$ when no assumption is made on the sign of $h$, we use the coupling between $N^h$ and $N^{h^+}$ stated in Proposition \ref{prop:couplage},  \ref{prop:couplage-ii}).

\begin{prop}\label{prop:tauleqtau^+}
Let Assumption \ref{hyp:h} hold.
Let $A<\infty$ be such that $A \ge L(h)$.
Let $\tau$ be defined in \eqref{def:tau}, and  $\tau^+$ be defined similarly with $h^+$ instead of $h$. Then
$\P_{\mathfrak{m}}(\tau \le \tau^+)=1$.
\end{prop}

\begin{proof}
We use the coupling $(N^h,N^{h^+})$ of  Proposition~\ref{prop:couplage},~\ref{prop:couplage-ii}), which satisfies
$N^h \le N^{h^+}$.
If $\tau=+\infty$, since the immigration rate $\lambda$ is positive, for {any} $t\ge0$ necessarily $N^h(t-A,t]>0$ and thus $N^{h^+}(t-A,t]>0$, which implies that $\tau^+=+\infty$ also, a.s.

Now, it is enough to prove that $\tau\leq \tau^+$ when both times are finite. In this case, since $N^{h^+}$ is locally finite a.s., $\tau^+-A$ is an atom of $N^{h^+}$ such that $N^{h^+}(\tau^+-A,\tau^+]=0$. This implies that $N^{h}(\tau^+-A,\tau^+]=0$.
If $\tau^+-A $ is also an atom of $N^h$, then $\tau\leq \tau^+$.

Else, first prove that $N^h(-A,\tau^+-A)>0$. The result is obviously true if $N^0\not= \emptyset$. When $N^0=\emptyset$, the first atoms of $N^h$ and $N^{h^+}$ coincide because $\Lambda^h_0=\Lambda_0^{h^+}$, where these functions are defined in \eqref{def:Lambda0}. This first atom is necessarily before $\tau^+-A$, and hence $N^h(-A,\tau^+-A)>0$.
The last atom $U$ of $N^h$ before $\tau^+-A$ is thus well defined, and necessarily satisfies $N^h(U,U+A]=0$ and $N^h [U, U+A)\neq 0$ so that
$\tau\leq U+A\leq \tau^+$.  We have thus proved that $\tau\leq \tau^+$,  $\P_{\mathfrak{m}}$-a.s.,
as needed.
\end{proof}

We now prove that the regeneration time $\tau$ admits an exponential moment which ensures that it is finite almost surely. The results will rely on the coupling between $N^h$ and $N^{h^+}$ and on the results obtained in Section \ref{subsec:cluster}. Let us define
\begin{equation*}
\gamma^+\triangleq\frac{\|h^+\|_1-\log(\|h^+\|_1)-1}{L(h^+)}>0\,.
\end{equation*}

\begin{prop}\label{prop:tau}
Let Assumption \ref{hyp:h} hold.
Let $A<\infty$ be such that $A \ge L(h)$, and assume that $\E_\mathfrak{m}(N^0(-A,0])<+\infty$.
Then $\tau$ given by \eqref{def:tau} satisfies
\[
\forall \alpha < \min(\lambda,{\gamma^+})\,,
\quad
\E_{\mathfrak{m}}(\mathrm{e}^{\alpha \tau}) < +\infty\,.
\]
In particular $\tau$ is finite, $\P_{\mathfrak{m}}$-a.s., and  $\E_{\mathfrak{m}}(\tau) < +\infty$.
\end{prop}

\begin{proof}
Using Proposition \ref{prop:tauleqtau^+}, it is sufficient to prove this for $\tau^+$.
When $\mathfrak{m}$ is the Dirac measure at $\emptyset$, the result is a direct consequence of
 Lemma \ref{lem:egalite_temps} and Proposition \ref{prop:decrR_1}.
 We now turn to the case when $\mathfrak{m}$ is different from $\delta_{\emptyset}$.
The proof is separated in three steps.

\paragraph{Step 1: Analysis of the problem.}
 To control $\tau^+$, we distinguish the points of $N^h$ coming from the initial condition from the points coming from ancestors arrived after zero. We thus introduce $K= N^0((-A,0])$, the number  of atoms of $N^0$, $(V_i^0)_{1\le i \le K}$, these atoms, and $(\widetilde{H}_i^0(A))_{1\le i \le K}$ the durations such that $V_i^0 + \widetilde{H}_i^0(A)-A$ is the time of birth of the last descendant of $V_i^0$. Note that $V_i^0$ has no offspring before time~$0$, so that the reproduction function of $V_i^0$ is a truncation of $h$. We finally define the time when the influence of the past before $0$ has vanished,
 given by
 $$E=\max_{1 \le i \le K}(V_i^0+\widetilde{H}_i^0(A)),$$
 with the convention that $E=0$ if $K=0$. If $K>0$, since $V_i^0\in (-A,0]$ and $\widetilde{H}_i^0(A)\geq A$, we have $E>0$. Note that $\tau^+ \ge E$. \\
 We now consider the sequence $(V_i)_{i \ge 1}$ of ancestors arriving after time $0$ at rate $\lambda$.  We recall that they can be viewed as the arrival of customers in a $M/G/\infty$ queue with time service of law that of $\widetilde{H}_1(A)$. In our case, the queue may not be empty at time $0$, when $E>0$. In that case, the queue returns to $0$ when all the customers arrived before time $0$ have left the system (which is the case at time $E$) and when all the busy periods containing the customers arrived at time between $0$ and $E$ are over. The first hitting time of $0$ for the queue is thus equal to
 \begin{equation}\label{reecriture:tau+}
 \tau^+ =  \left\{ \begin{array}{ccl}
 E & \mbox{ if } & Y_E=0\,,
 \\
 \inf\{t \ge E :  Y_t=0 \} & \mbox{ if } & Y_E>0\,,
 \end{array}\right.
 \end{equation}
where $Y_t$ is given in \eqref{def:queue} by $Y_t=\sum_{k : 0\leq V_k\le t} \ind_{\{V_k+\widetilde{H}_k(A)>t\}}.$

\paragraph{Step 2: Exponential moments of $E$.}
In \eqref{reecriture:tau+}, $E$ depends only on $N^0$ and $(Y_t)_{t\geq 0}$ only on the arrivals and service times of customers entering the queue after time $0$. A natural idea is then to condition with respect to $E$, and for this it is important to gather estimates on the moments of $E$.
Since $V_i^0\leq 0$, we have that
$$0\leq E\leq \max_{1\leq i\leq K} \widetilde{H}^0_i(A).$$
The truncation mentioned in Step 1 implies that the $\widetilde{H}^0_i(A)$ are stochastically dominated by independent random variables distributed as $\widetilde{H}_1$, which we denote by $\bar{H}^0_i(A)$. Thus for $t>0$,
using \eqref{eq:queueH1tilde},
\begin{align*}
\P_{\mathfrak{m}}(E>t)
&\leq \P_{\mathfrak{m}}\big(\max_{1\leq i\leq K} \bar{H}^0_i(A)>t\big)\\
&=  1-\E_{\mathfrak{m}}\Big(\big(1-\P(\widetilde{H}_1(A)> t)\big)^K\Big)\\
&\leq  1 -\E_{\mathfrak{m}}\big((1-C \mathrm{e}^{-{\gamma^+} t})^K\big)\,.
\end{align*}
Thus there exists $t_0>0$ such that for any $t>t_0$,
\begin{align}
\P_{\mathfrak{m}}(E>t) \leq  C \E_\mathfrak{m}(N^0(-A,0]) \mathrm{e}^{-{\gamma^+}t}.\label{eq:controleE}
\end{align}As a corollary, we have for any $\beta \in (0,{\gamma^+})$ that
\begin{equation}\label{eq:momentexpE}
\E_\mathfrak{m}\big(\mathrm{e}^{\beta E}\big)<+\infty\,.
\end{equation}

\paragraph{Step 3: Estimate of the tail distribution of $\tau^+$.}
For $t>0$, we have
\begin{align*}
\P_{\mathfrak{m}}(\tau^+>t)& = \P_{\mathfrak{m}}\big(\tau^+>t,\,E> t\big)+\P_{\mathfrak{m}}\big(\tau^+>t,\,E\leq t\big)\\
& \leq\P_{\mathfrak{m}}(E> t)+\E_{\mathfrak{m}}\Big(\ind_{\{E\leq t\}}\, \P_{\mathfrak{m}}\big(\tau^+>t \,|\, E\big)\Big).
\end{align*}The first term is controlled by \eqref{eq:controleE}. For the second term, we use Proposition \ref{prop:appendice-suitecarl} which is a consequence of Theorem \ref{thm:carl}. For this, let us introduce a constant $\kappa$ such that $\kappa <{\gamma^+}$ if ${\gamma^+} \leq \lambda$ and $\kappa=\lambda$ if $\lambda<{\gamma^+}$. We have:
\begin{equation*}
\E_{\mathfrak{m}}\Big(\ind_{\{E\leq t\}} \, \P\big(\tau^+>t \,|\, E\big)\Big) \leq \E_{\mathfrak{m}}\big(\ind_{\{E\leq t\}} \, \lambda  C E\,\mathrm{e}^{-\kappa(t-E)}\big)=  \lambda  C \mathrm{e}^{-\kappa t} \E_\mathfrak{m}\big(\ind_{\{E\leq t\}} \, E\,\mathrm{e}^{\kappa E}\big).
\end{equation*}
Since $\kappa<{\gamma^+}$, it is always possible to choose $\beta\in (\kappa,{\gamma^+})$ such that \eqref{eq:momentexpE} holds, which entails that
$\E_\mathfrak{m}\big(\ind_{\{E\leq t\}} \,E\,\mathrm{e}^{\kappa E}\big)$ can be bounded by a finite constant independent of $t$.

Gathering all the results,
\begin{align*}
\P_{\mathfrak{m}}(\tau^+>t) & \leq C \E_\mathfrak{m}(N^0(-A,0]) \mathrm{e}^{-{\gamma^+}t}+ \lambda C' \mathrm{e}^{-\kappa t}=O\big(\mathrm{e}^{-\kappa t}\big).
\end{align*}This yields that $\E_\mathfrak{m}(\mathrm{e}^{\alpha \tau^+})<+\infty$ for any $\alpha<\kappa$, i.e. $\alpha<\min (\lambda,{\gamma^+})$.
 \end{proof}

Note that if Assumption \ref{hyp:h} holds then
$\tau$ given by~\eqref{def:tau} satisfies $\E_\emptyset(\tau)<\infty$
and hence the null measure $\emptyset$ is a positive recurrent state for the strong Markov process $X=(X_t)_{t\ge 0}$.

\begin{thm}
\label{thm:exist-uniq-inv-law}
Let Assumption \ref{hyp:h} hold.
The strong Markov process $X=(X_t)_{t\ge 0}$  with values in $\mathcal{N}((-A,0])$
defined by \eqref{aux-proc} admits a unique  invariant law $\pi_A$  defined as in~\eqref{def:piA}, \emph{i.e.},
for every Borel nonnegative function $f$ on $\mathcal{N}((-A,0])$,
\[
\pi_A f  = \frac1{\E_\emptyset(\tau)} \E_\emptyset\biggl(\int_0^{\tau} f(X_t) \,dt\biggr)\,.
\]
Moreover  $\pi_A\{\emptyset\} = 1/(\lambda \E_\emptyset(\tau))$.
\end{thm}

\begin{proof}
These facts are classic in the presence of the positive recurrent state $\emptyset$ which is reachable from all states.
\end{proof}

The strong Markov property of $X$  yields a  sequence of regeneration times $(\tau_k)_{k\geq 0}$,
which are the successive visits of $X$ to the positive recurrent state $\emptyset$,
defined as follows (the time $\tau_0$ has already been introduced in \eqref{def:tau0}):
\begin{align*}
\tau_0 &=  \inf\{t\ge 0 : X_t =\emptyset\}\,,
&&\text{(First entrance time of $\emptyset$)}
\\
\tau_k &= \inf\{t>\tau_{k-1} : X_{t-}\neq \emptyset, X_{t} =\emptyset\}\,,
\quad k\ge1\,.
&&\text{(Successive return times at $\emptyset$)}
\end{align*}
They provide a useful decomposition of the path of $X$ in i.i.d.\ excursions.

\begin{thm}
\label{thm:reg-seq}
Let $N^h$ be a Hawkes process satisfying Assumption~\ref{hyp:h}, and  $A \ge L(h)$.
Consider the Markov process $X $ defined in \eqref{aux-proc}. Under $\P_\mathfrak{m}$ the following holds:
\begin{enumerate}[\rm a)]
\item
\label{tau-k-finite}
The $\tau_k$ for $k\ge0$ are finite stopping times, a.s.
\item
\label{delay-ind-cycle}
The delay $(X_{t})_{ t \in [0,\tau_0) }$ is independent of the cycles $(X_{\tau_{k-1} + t})_{ t \in [0, \tau_k- \tau_{k-1}) }$ for $k\ge1$.
\item
\label{cycles-iid}
These cycles are i.i.d.\ and distributed as $(X_t)_{t \in[0, \tau)}$ under $\P_\emptyset$.
In particular their durations $(\tau_k-\tau_{k-1})_{k\geq 1}$ are distributed as $\tau$ under $\P_\emptyset$,
and $\lim_{k\rightarrow +\infty}\tau_k=+\infty$, $\P_\mathfrak{m}$-a.s.
\end{enumerate}

\end{thm}

\begin{proof}
The above items follow classically from the strong Markov property of $X$.
Let us first prove the finiteness of the return times $\tau_k$. For any $\mathfrak{m}$, from the definition of $\tau_0$ and $\tau$, we have that $\tau_0\leq \tau$, $\P_\mathfrak{m}$-a.s. Then, $\P_\mathfrak{m}(\tau_0<+\infty)=1$ follows from Proposition \ref{prop:tau} (i).
For $k\geq 1$, using the strong Markov property of $X$, we have for any $\mathfrak{m}$:
\begin{align*}
\P_\mathfrak{m}(\tau_k<+\infty)
&=   \E_\mathfrak{m} \big( \ind_{\{\tau_{k-1}<+\infty\}} \, \P_{X_{\tau_{k-1}}}(\tau<+\infty)\big)\\
 &=  \E_\mathfrak{m}\big(\ind_{\{\tau_{k-1}<+\infty\}} \,\P_\emptyset(\tau<+\infty)\big)\\
& =  \P_\mathfrak{m}(\tau_{k-1}<+\infty)=\cdots = \P_{\mathfrak{m}}(\tau_0<+\infty)=1.
\end{align*}

Let us now prove \ref{delay-ind-cycle}) and \ref{cycles-iid}).
It is sufficient to consider $(X_t)_{t\in [0,\tau_0)}$, $(X_{\tau_0+t})_{t\in [0,\tau_1-\tau_0)}$ and $(X_{\tau_1+t})_{t\in [0,\tau_2-\tau_1)}$. Let $F_0$, $F_1$, and $F_2$ be three measurable bounded real functions on $\D(\R_+,\mathcal{N}(-A,0])$. Then, using the strong Markov property successively at $\tau_1$ and $\tau_0$, we obtain:
\begin{align*}
&\E_{\mathfrak{m}} \Big(F_0\big((X_t)_{t\in [0,\tau_0)}\big)\, F_1\big((X_{\tau_0+t})_{t\in [0,\tau_1-\tau_0)}\big) \,F_2\big((X_{\tau_1+t})_{t\in [0,\tau_2-\tau_1)}\big)\Big)
\\
&\quad = \E_{\mathfrak{m}} \Big(F_0\big((X_t)_{t\in [0,\tau_0)}\big)\Big) \,\E_{\emptyset}\Big(F_1\big((X_{t})_{t\in [0,\tau)}\big)\Big) \,\E_\emptyset\Big( F_2\big((X_{t})_{t\in [0,\tau)}\big)\Big).
\end{align*}
This concludes the proof.
\end{proof}

\section{Proof of the main results}\label{sec:proofs}

We translate the statements of the main results in terms of the Markov process $X$.
Let $T>0$ be fixed and, since the sequence $(\tau_k)_{k\geq 0}$ increases to infinity,
\begin{equation}
\label{def:kt}
 K_T\triangleq \max\{k\ge0 : \tau_k\le T\} \xrightarrow[T\to\infty]{\P_\mathfrak{m}-\text{a.s.}} \infty\,.
\end{equation}
For a locally bounded Borel function $f$ on $\mathcal{N}((-A,0])$ we define the random variables
\begin{equation}
\label{integrals}
I_k f \triangleq \int_{\tau_{k-1}}^{\tau_k} f (X_t)\,dt\,,
\quad
k\ge1\,,
\end{equation}
which are finite a.s., i.i.d., and of the same law as $\int_{0}^{\tau} f (X_t)\,dt$ under $\P_\emptyset$,
see Theorem~\ref{thm:reg-seq}.

\subsection*{Proof of Theorem~\ref{thm:point-erg+laws}~\ref{lln})}

This classic proof assumes first that $f\ge0$. Then using~\eqref{def:kt} and~\eqref{integrals},
\begin{equation*}
\frac1{K_T}\sum_{k=1}^{K_T} I_k f
\le \frac1{K_T}\int_0^T f(X_t)\,dt
\le
\frac1{K_T}\int_0^{\tau_0} f(X_t)\,dt + \frac1{K_T}\sum_{k=1}^{K_T+1} I_kf
\end{equation*}
and the strong law of large numbers applied to the i.i.d.~$I_k f$ yields that
\begin{equation*}
\frac1{K_T}\int_0^T f(X_t)\,dt \xrightarrow[T\to\infty]{\P_\mathfrak{m}-\text{a.s.}}
\E_\emptyset\biggl(\int_0^{\tau} f(X_t) \,dt\biggr) \triangleq \E_\emptyset(\tau) \,\pi_A f ~.
\end{equation*}
Choosing $f=1$ yields that
\begin{equation}
\label{t-et-kt}
\frac{T}{K_T} \xrightarrow[T\to\infty]{\P_\mathfrak{m}-\textnormal{a.s.}} \E_\emptyset(\tau) <\infty
\end{equation}
and dividing the first limit by the second concludes the proof for $f\ge0$.
The case of $\pi_A$-integrable signed $f$ follows using the decomposition $f=f^+ - f^-$.

%

\subsection*{Proof of Theorem~\ref{thm:point-erg+laws}~\ref{cv-to-equ})}

This follows from a general result in Thorisson~\cite[Theorem~10.3.3~p.351]{Thorisson2000},
which yields that if the distribution of $\tau$ under $\P_\emptyset$ has a density with respect to the Lebesgue measure
and if $\E_\emptyset(\tau)<+\infty$, then there exists a probability measure $\mathbb{Q}$ on $\D(\R_+,\mathcal{N}(-A,0])$ such that,
for any initial law $\mathfrak{m}$,
\begin{equation*}
\P_\mathfrak{m}\bigl((X_{t+u})_{u\ge0}\in \cdot \bigr)
\xrightarrow[t\to\infty]{\textnormal{total variation}}
\mathbb{Q}\,.
\end{equation*}
Since $\pi_A$ is an invariant law,
$ \P_{\pi_A}\bigl((X_{t+u})_{u\ge0}\in \cdot \bigr) = \P_{\pi_A}(X\in \cdot) $ for every $t\geq 0$. Hence, taking $\mathfrak{m}=\pi_A$ in the above convergence yields that $\mathbb{Q}=\P_{\pi_A}(X\in \cdot)$.

It remains to check the above assumptions of the theorem.
Proposition \ref{prop:tau} yields that  $\E_\emptyset(\tau)<+\infty$.
Moreover, under $\P_{\emptyset}$ we can rewrite $\tau$ as
$$\tau=U_1^h+\inf\big\{t>0 : \ X_{(t+U_1^h)_-}\not= \emptyset \mbox{ and } X_{t+U_1^h}= \emptyset\big\}.$$
Using the strong Markov property, we easily prove independence of the two terms in the r.h.s.
Since $U_1^h$ has an exponential distribution under $\P_{\emptyset}$, $\tau$ has a density under $\P_\emptyset$.

\subsection*{Proof of Theorem~\ref{thm:clt}}

Let $\tilde{f} \triangleq f -\pi_A f$,
so that $\frac1T \int_0^T  \tilde{f} (X_t)\,dt = \frac1T \int_0^T  f(X_t)\,dt  - \pi_A f $.
With the notation~\eqref{def:kt} and~\eqref{integrals},
we have the decomposition
\begin{equation}\label{eq:decompo}
\int_0^T \tilde{f}(X_t)\,dt
= \int_0^{\tau_0} \tilde{f}(X_t)\,dt
+ \sum_{k=1}^{K_T} I_k \tilde{f} + \int_{\tau_{K_T}}^T \tilde{f}(X_t)\,dt\,.
\end{equation}
The $I_k \tilde{f}$ are i.i.d. and
are distributed as $\int_{0}^{\tau} \tilde{f} (X_t)\,dt$ under $\P_\emptyset$,
with expectation~$0$ and variance  $\E_\emptyset(\tau) \sigma^2(f)$,
see Theorem~\ref{thm:reg-seq}.
Since $f$ is locally bounded, so is $\tilde{f}$ and
\[
\frac1{\sqrt{T}} \int_0^{\tau_0} \tilde{f}(X_t)\,dt  \xrightarrow[T\to\infty]{\P_\mathfrak{m}-\text{a.s.}} 0\,.
\]
Now, let $\eps >0$. For arbitrary $a>0$ and $0<u\le T$,
\[
\P_\mathfrak{m}\biggl(\biggl|\int_{\tau_{K_T}}^T \tilde{f}(X_t)\,dt \biggr| > a\biggr)
\le
\P_\mathfrak{m}(T-\tau_{K_T}> u) +
\P_\mathfrak{m}\biggl(\sup_{0\le s \le u}\biggl|\int_{T-s}^T \tilde{f}(X_t)\,dt \biggr| > a\biggr)\,.
\]
But
\[
\P_\mathfrak{m}(T-\tau_{K_T}> u)= 1 - \P_\mathfrak{m}(\exists t \in [T-u,T] : X_{t-}\neq \emptyset, X_t=\emptyset)
\]
and Theorem~\ref{thm:point-erg+laws}~\ref{cv-to-equ}) yields that
\[
\lim_{T\to\infty}\P_\mathfrak{m}(T-\tau_{K_T}> u)
= 1-\P_{\pi_A}(\exists t \in [0,u] : X_{t-}\neq \emptyset, X_t=\emptyset) ~,
\]
so that  there exists $u_0$ large enough  such that
\[
\lim_{T\to\infty} \P_\mathfrak{m}(T-\tau_{K_T}> u_0) <\frac{\eps}{2}\,.
\]
Moreover Theorem~\ref{thm:point-erg+laws}~\ref{cv-to-equ}) yields that
\[
\lim_{T\to\infty}\P_\mathfrak{m}\biggl(\sup_{0\le s \le u_0}\biggl|\int_{T-s}^T \tilde{f}(X_t)\,dt \biggr| > a\biggr)
=\P_{\pi_A}\biggl(\sup_{0\le s \le u_0}\biggl|\int_0^s \tilde{f}(X_t)\,dt \biggr| > a\biggr)
\]
and thus there exists $a_0$ large enough such that
\[
\lim_{T\to\infty}\P_\mathfrak{m}\biggl(\sup_{0\le s \le u_0}\biggl|\int_{T-s}^T \tilde{f}(X_t)\,dt \biggr| > a_0\biggr) <\frac{\eps}{2}
\]
and hence
\[
\limsup_{T\to\infty} \P_\mathfrak{m}\biggl(\biggl|\int_{\tau_{K_T}}^T \tilde{f}(X_t)\,dt \biggr| > a_0\biggr) < \eps\,.
\]
This implies in particular that
\[
\frac1{\sqrt{T}} \int_{\tau_{K_T}}^T  \tilde{f}(X_t)\,dt \xrightarrow[T\to\infty]{\text{probab.}} 0\,.
\]
It now remains to treat the second term in the r.h.s. of \eqref{eq:decompo}. The classic central limit theorem yields that
\[
\frac1{\sqrt{T}} \sum_{k=1}^{\lfloor T/\E_\emptyset(\tau)\rfloor} I_k\tilde{f}
\xrightarrow[T\to\infty]{\textnormal{in law}}
\frac1{\sqrt{\E_\emptyset(\tau)}} \mathcal{N}(0, \E_\emptyset(\tau)\sigma^2(f)) = \mathcal{N}(0, \sigma^2(f))
\]
and we are left to control
\[
\Delta_T \triangleq
\frac1{\sqrt{T}}\sum_{k=1}^{K_T} I_k\tilde{f}
- \frac1{\sqrt{T}}\sum_{k=1}^{\lfloor T/\E_\emptyset(\tau)\rfloor} I_k\tilde{f}\,.
\]
Let $\eps>0$ and
\[
v(T,\eps) \triangleq \{\lfloor (1 -\eps^3)T/\E_\emptyset(\tau)\rfloor, \dots,
\lfloor (1 +\eps^3)T/\E_\emptyset(\tau)\rfloor\}\,.
\]
Note that $(1-\eps^3)T/\E_\emptyset(\tau)<T/\E_\emptyset(\tau)<(1+\eps^3)T/\E_\emptyset(\tau)$
and hence that
$\lfloor T/\E_\emptyset(\tau)\rfloor$ belongs to $v(T,\eps) $.
In view of~\eqref{t-et-kt}, there exists $t_\eps$ such that if $T\ge t_\eps$
\[
\P_\mathfrak{m}(K_T \in v(T,\eps)) >1-\eps\,.
\]
For $T\ge t_\eps$ we thus have on $\{K_T \in v(T,\eps)\}$ that
\begin{align*}
|\Delta_T |
& \leq  \Biggl| \frac{1}{\sqrt{T}} \sum_{k=\lfloor (1 -\eps^3)T/\E_\emptyset(\tau)\rfloor}^{K_T} I_k\tilde{f} \Biggr|
+ \Biggl| \frac{1}{\sqrt{T}}
\sum_{k=\lfloor (1 -\eps^3)T/\E_\emptyset(\tau)\rfloor}^{\lfloor T/\E_\emptyset(\tau)\rfloor} I_k\tilde{f} \Biggr|
\\
& \leq    \frac{2}{\sqrt{T}}
\max_{n\in v(T,\eps)}
\Biggl|
\sum_{k=\lfloor (1 -\eps^3)T/\E_\emptyset(\tau)\rfloor}^n I_k\tilde{f}
\Biggr|\, .
\end{align*}
Using now Kolmogorov's maximal inequality~\cite[Sect.~IX.7 p.234]{Feller1968} we obtain that
\[
\P_\mathfrak{m}(|\Delta_T | \ge \eps)
\le \frac{\lfloor (1 +\eps^3)T/\E_\emptyset(\tau)\rfloor - \lfloor (1 -\eps^3)T/\E_\emptyset(\tau)\rfloor}{\eps^2 T/4} \E_\emptyset(\tau) \sigma^2(f) \le 8\sigma^2(f) \eps \,.
\]
Since $\eps >0$ is arbitrary, we conclude that
\[
\Biggl| \frac1{\sqrt{T}}\sum_{k=1}^{K_T} I_k\tilde{f}
- \frac1{\sqrt{T}}\sum_{k=1}^{\lfloor T/\E_\emptyset(\tau)\rfloor} I_k\tilde{f}
\Biggr| \xrightarrow[T\to\infty]{\text{probab.}} 0\,.
\]
These three convergence results and Slutsky's theorem yield the convergence result.

\subsection*{Proof of Theorem~\ref{thm:non-asympt-exp-bd}}

With the notations $\tilde{f} \triangleq f -\pi_A f$,
so that $\frac1T \int_0^T  \tilde{f} (X_t)\,dt = \frac1T \int_0^T  f(X_t)\,dt  - \pi_A f $,
and~\eqref{integrals}, let us consider the decomposition
\begin{align}
\label{sum-conc-ineq}
\int_0^T \tilde{f} (X_t)\,dt
=
\int_0^{\tau_0}  \tilde{f}(X_t)\,dt
+ \sum_{k=1}^{\lfloor T/\E_\emptyset(\tau)\rfloor} I_k \tilde{f}
+ \int_{\tau_{\lfloor T/\E_\emptyset(\tau)\rfloor}}^T  \tilde{f}(X_t)\,dt \,.
\end{align}
The $I_k \tilde{f}$ are i.i.d.~and distributed as $\int_{0}^{\tau} \tilde{f} (X_t)\,dt$ under $\P_\emptyset$,
with expectation~$0$ and variance  $\E_\emptyset(\tau) \sigma^2(f)$,
see Theorem~\ref{thm:reg-seq}.
Since $f$ takes its values in $[a,b]$,
\[
\biggl|
\int_0^{\tau_0} \tilde{f} (X_t)\,dt
\biggr|
\le |b-a| \tau_0 \,,
\qquad
\biggl|
\int_{\tau_{\lfloor T/\E_\emptyset(\tau)\rfloor}}^T \tilde{f} (X_t)\,dt
\biggr|
\le |b-a||T-\tau_{\lfloor T/\E_\emptyset(\tau)\rfloor}|\,.
\]
Now
\begin{align*}
T-\tau_{\lfloor T/\E_\emptyset(\tau)\rfloor}
&= - \tau_0 - \sum_{k=1}^{\lfloor T/\E_\emptyset(\tau)\rfloor} (\tau_k-\tau_{k-1}) +T
\\
&= - \tau_0 -\sum_{k=1}^{\lfloor T/\E_\emptyset(\tau)\rfloor}  (\tau_k-\tau_{k-1}-\E_\emptyset(\tau))  + T -\lfloor T/\E_\emptyset(\tau)\rfloor \E_\emptyset(\tau)\end{align*}
in which
$
0\le T-\lfloor T/\E_\emptyset(\tau)\rfloor \E_\emptyset(\tau)  < \E_\emptyset(\tau)
$
and the $\tau_k-\tau_{k-1}-\E_\emptyset(\tau)$ are i.i.d., have same law as $\tau-\E_\emptyset(\tau) $
under $\P_\emptyset$, and have expectation $0$ and variance $\Var_\emptyset (\tau)$.
Thus,
\begin{align*}
&\P_\mathfrak{m}\biggl( \biggl|\frac1T \int_0^T  f(X_t)\,dt - \pi_A f \biggr| \ge \varepsilon \biggr)
\\
&\le \P_\mathfrak{m}\left( \left| \sum_{k=1}^{\lfloor T/\E_\emptyset(\tau)\rfloor} I_k \tilde{f} \right| + |b-a|
 \left( 2\tau_0 +\left|\sum_{k=1}^{\lfloor T/\E_\emptyset(\tau)\rfloor} (\tau_k-\tau_{k-1}-\E_\emptyset(\tau))\right|
+ \E_\emptyset(\tau)\right)\ge T\eps
\right)\,.
\end{align*}
Now, using that for any $u \in [0,1)$
$$T\eps  -|b-a| \E_\emptyset(\tau) -2|b-a| \E_\mathfrak{m}(\tau_0)
= 2 \frac{(1-u)T\eps - |b-a| \E_\emptyset(\tau)}2 + uT\eps -2|b-a| \E_\mathfrak{m}(\tau_0)~,$$we obtain that
\begin{align}
\label{dec-proba}
&\P_\mathfrak{m}\biggl( \biggl|\frac1T \int_0^T  f(X_t)\,dt - \pi_A f \biggr| \ge \varepsilon \biggr)
\notag\\
&\quad\le
\P_\mathfrak{m}\left(
\left|
 \sum_{k=1}^{\lfloor T/\E_\emptyset(\tau)\rfloor} I_k \tilde{f}
 \right|
\ge \frac{(1-u)T\eps - |b-a| \E_\emptyset(\tau)}2
\right)
\notag\\
&\qquad
+
\P_\mathfrak{m}\left(
\left|
\sum_{k=1}^{\lfloor T/\E_\emptyset(\tau)\rfloor} (\tau_k-\tau_{k-1}-\E_\emptyset(\tau))
\right|
\ge \frac{(1-u)T\eps- |b-a| \E_\emptyset(\tau)}{2|b-a|}
\right)
\notag\\
&\qquad+
\P_\mathfrak{m}\left(
\tau_0-\E_\mathfrak{m}(\tau_0)
\ge \frac{uT\eps  - 2|b-a| \E_{\mathfrak{m}}(\tau_0)}{2|b-a|}
\right).
\end{align}
We aim to apply Bernstein's inequality~\cite[Cor.~2.10~p.25, (2.17), (2.18)~p.24]{massart_concentration} {to bound the three terms of the right hand side. We recall that to apply Bernstein inequality to random variables $X_1,\dots X_N$, there should exist constants $c$ and $v$ such that
$$ \sum_{k=1}^N\E_\mathfrak{m}\left[X_k^2\right]\le v,\quad \text{ and } \quad \sum_{k=1}^N\E_\mathfrak{m}\left[(X_k)_+^n\right]\le \frac{n!}{2}vc^{n-2},\quad \forall n\ge3.$$
}First,
\[
\sum_{k=1}^{\lfloor T/\E_\emptyset(\tau)\rfloor} \E_\mathfrak{m}\bigl( ( I_k \tilde{f})^2 \bigr)
=  \Bfl\frac{T}{\E_\emptyset(\tau)}\Bfr \E_\emptyset(\tau) \sigma^2(f)
\le  T \sigma^2(f)
\]
and, for $n\ge3$,
\begin{multline*}
\sum_{k=1}^{\lfloor T/\E_\emptyset(\tau)\rfloor} \E_\mathfrak{m}\bigl( (I_k \tilde{f})_\pm^n \bigr)
=  \Bfl \frac{T}{\E_\emptyset(\tau)}\Bfr  \E_\mathfrak{m}\bigl( (I \tilde{f})_\pm^n\bigr)
\\
 \le
\frac{n!}2  T \sigma^2(f)
\biggl(
\sup_{k\ge3}
\biggl(
\frac2{k!}\frac{\E_\mathfrak{m}\bigl( (I\tilde{f})_\pm^k \bigr)}{\E_\emptyset(\tau)\sigma^2(f)}
\biggr)^{\frac1{k-2}}
\biggr)^{n-2}
\triangleq
\frac{n!}2  T \sigma^2(f)(c^\pm(f))^{n-2}\,.
\end{multline*}
Then,
\[
\sum_{k=1}^{\lfloor T/\E_\emptyset(\tau)\rfloor} \E_\mathfrak{m}\bigl( (\tau_k-\tau_{k-1}-\E_\emptyset(\tau))^2 \bigr)
=  \Bfl\frac{T}{\E_\emptyset(\tau)}\Bfr \Var_\emptyset (\tau)
\le  T \frac{\Var_\emptyset (\tau)}{\E_\emptyset(\tau)}
\]
and, for $n\ge3$,
\begin{multline*}
\sum_{k=1}^{\lfloor T/\E_\emptyset(\tau)\rfloor} \E_\mathfrak{m}\bigl( (\tau_k-\tau_{k-1}-\E_\emptyset(\tau))_\pm^n \bigr)
 =\Bfl T/\E_\emptyset(\tau)\Bfr \E_\emptyset\bigl( (\tau-\E_\emptyset(\tau))_\pm^n \bigr)
\\
\le
\frac{n!}2  T \frac{\Var_\emptyset (\tau)}{\E_\emptyset(\tau)}
\biggl(
\sup_{k\ge3}
\biggl(
\frac2{k!}
\frac{\E_\emptyset\bigl( (\tau-\E_\emptyset(\tau))_\pm^k \bigr)}%
{\Var_\emptyset (\tau)}
\biggr)^{\frac1{k-2}}
\biggr)^{n-2}
\triangleq
\frac{n!}2  T \frac{\Var_\emptyset (\tau)}{\E_\emptyset(\tau)}(c^\pm(\tau))^{n-2}\,.
\end{multline*}
Lastly, $\E_\mathfrak{m}\bigl((\tau_0-\E_\mathfrak{m}(\tau_0))^2\bigr) = \Var_\mathfrak{m}(\tau_0)$ and, for $n\ge3$,
\begin{align*}
&\E_\mathfrak{m}\bigl((\tau_0-\E_\mathfrak{m}(\tau_0))_+^n\bigr)
\\
&\quad \le \frac{n!}2 \Var_\mathfrak{m}(\tau_0)
\biggl(
\sup_{k\ge3}
\biggl(
\frac2{k!}
\frac{\E_\mathfrak{m}\bigl((\tau_0-\E_\mathfrak{m}(\tau_0))_+^k\bigr)}%
{\Var_\mathfrak{m}(\tau_0)}
\biggr)^{\frac1{k-2}}
\biggr)^{n-2}
\triangleq
\frac{n!}2 \Var_\mathfrak{m}(\tau_0)(c^+(\tau_0))^{n-2}\,.
\end{align*}

Applying~\cite[Cor.~2.10 p.25, (2.17), (2.18)~p.24]{massart_concentration}
to the r.h.s. of~\eqref{dec-proba} yields that
\begin{align*}
&\P_\mathfrak{m}\biggl( \biggl|\frac1T \int_0^T  f(X_t)\,dt - \pi_A f \biggr| \ge \varepsilon \biggr)
\notag\\
&\quad\le
\exp\left(
-\frac{((1-u)T\eps - |b-a| \E_\emptyset(\tau))^2}
{8 T \sigma^2(f) + 4 c^+(f)((1-u)T\eps - |b-a| \E_\emptyset(\tau)) }
\right)
\notag\\
&\qquad+
\exp\left(
-\frac{((1-u)T\eps - |b-a| \E_\emptyset(\tau))^2}
{8 T \sigma^2(f) + 4 c^-(f)((1-u)T\eps - |b-a|\E_\emptyset(\tau)) }
\right)
\notag\\
&\qquad+
\exp\left(
-\frac{((1-u)T\eps - |b-a| \E_\emptyset(\tau))^2}
{8T |b-a|^2\frac{\Var_\emptyset (\tau)}{\E_\emptyset(\tau)}+ 4 |b-a| c^+(\tau)((1-u)T\eps - |b-a| \E_\emptyset(\tau)) }
\right)
\notag\\
&\qquad+
\exp\left(
-\frac{((1-u)T\eps - |b-a| \E_\emptyset(\tau))^2}
{8T |b-a|^2\frac{\Var_\emptyset (\tau)}{\E_\emptyset(\tau)}+ 4 |b-a| c^-(\tau)((1-u)T\eps - |b-a| \E_\emptyset(\tau)) }
\right)
\notag\\
&\qquad+
\exp\left(
-\frac{(uT\eps  - 2|b-a| \E_\mathfrak{m}(\tau_0))^2}
{8|b-a|^2\Var_\mathfrak{m} (\tau_0) + 4 |b-a| c^+(\tau_0)(uT\eps  - 2|b-a| \E_\mathfrak{m}(\tau_0)) }
\right)
\end{align*}
which is \eqref{eq:conc-ineq}.

\subsection*{Proof of Corollary \ref{cor:dev_simple}}

Under $\P_\emptyset$, $\tau_0=0$ and thus Equation \eqref{dec-proba} reads:
\begin{align}\label{etape6}
\P_\emptyset\biggl( \biggl|\frac1T \int_0^T  f(X_t)\,dt - \pi_A f \biggr| \ge \varepsilon \biggr)
&\le
\P_\emptyset\left(
\left|
 \sum_{k=1}^{\lfloor T/\E_\emptyset(\tau)\rfloor} I_k \tilde{f}
 \right|
\ge \frac{T\eps - |b-a| \E_\emptyset(\tau)}2
\right)
\\
&
+
\P_\emptyset\left(
\left|
\sum_{k=1}^{\lfloor T/\E_\emptyset(\tau)\rfloor} (\tau_k-\tau_{k-1}-\E_\emptyset(\tau))
\right|
\ge \frac{T\eps-   |b-a| \E_\emptyset(\tau)}{2|b-a|}
\right).\nonumber
\end{align}
Similarly as for the proof of Theorem \ref{thm:non-asympt-exp-bd} we apply Bernstein inequality for each of the terms in the right hand side. However, in order to simplify the obtained bound, we change the upper bounds of the moments of $I_k\tilde{f}$ and $\tau_k-\tau_{k-1}-\E_\emptyset(\tau)$. Namely we use the fact that for all $n\ge1$,
\[
\E_\emptyset(\tau^n)\le \frac{n!}{\alpha^n}\E_\emptyset(e^{\alpha\tau}) \quad \text{and} \quad \E_\emptyset(|\tau - \E_\emptyset(\tau)|^n)\le \frac{n!}{\alpha^n}\E_\emptyset(e^{\alpha\tau}) e^{\alpha \E_\emptyset(\tau)}.
\]
Since $\tau$ is a nonnegative random variable, $e^{\alpha \E_\emptyset(\tau)} \ge 1$ and in the sequel it will be more convenient to use the following upper bound: for all $n\ge1$,
\[
\E_\emptyset(\tau^n)\le \frac{n!}{\alpha^n}\E_\emptyset(e^{\alpha\tau})  e^{\alpha \E_\emptyset(\tau)}.
\]
Then
\[
\sum_{k=1}^{\lfloor T/\E_\emptyset(\tau)\rfloor} \E_\emptyset\bigl( ( I_k \tilde{f})^2 \bigr)
\le  \Bfl\frac{T}{\E_\emptyset(\tau)}\Bfr \E_\emptyset(\tau^2) (b-a)^2
\le  \frac{2(b-a)^2}{\alpha^2}\Bfl\frac{T}{\E_\emptyset(\tau)}\Bfr \E_\emptyset(e^{\alpha\tau}) e^{\alpha \E_\emptyset(\tau)}\,,
\]
and, for $n\ge3$,
\[
\sum_{k=1}^{\lfloor T/\E_\emptyset(\tau)\rfloor} \E_\emptyset\bigl( |I_k \tilde{f})|^n \bigr)
 \le \frac{n!}{2} \left(\Bfl\frac{T}{\E_\emptyset(\tau)}\Bfr  |b-a|^2 \frac{2}{\alpha^2}\E_\emptyset(e^{\alpha\tau})e^{\alpha \E_\emptyset(\tau)}\right) \ \Big(\frac{|b-a|}{\alpha}\Big)^{n-2}\,.
\]
Setting
\[
v=\frac{2(b-a)^2}{\alpha^2}\Bfl\frac{T}{\E_\emptyset(\tau)}\Bfr \E_\emptyset(e^{\alpha\tau}) e^{\alpha \E_\emptyset(\tau)},\quad \text{and}\quad c= \frac{|b-a|}{\alpha},
\]
and applying Bernstein inequality, we obtain that
\[
 \P_\emptyset\left(
\left|
 \sum_{k=1}^{\lfloor T/\E_\emptyset(\tau)\rfloor} I_k \tilde{f}
 \right|
\ge \frac{T\eps - |b-a| \E_\emptyset(\tau)}2
\right) \le 2 \exp\left(-\frac{\Bigl( T\varepsilon-|b-a|\E_\emptyset(\tau) \Bigr)^2}{4 \left(2v + (T\varepsilon-|b-a|\E_\emptyset(\tau))c \right)}  \right).
\]
Also
\[
\sum_{k=1}^{\lfloor T/\E_\emptyset(\tau)\rfloor} \E_\emptyset\bigl( (\tau_k-\tau_{k-1}-\E_\emptyset(\tau))^2 \bigr)
\le \frac{2}{\alpha^2} \Bfl \frac{T}{\E_\emptyset(\tau)}\Bfr  \E_\emptyset(e^{\alpha\tau})e^{\alpha \E_\emptyset(\tau)} \,,
\]
and, for $n\ge3$,
\[
\sum_{k=1}^{\lfloor T/\E_\emptyset(\tau)\rfloor} \E_\emptyset\bigl( |\tau_k-\tau_{k-1}-\E_\emptyset(\tau)|^n \bigr)
\le \frac{n!}{2}\left(\Bfl\frac{T}{\E_\emptyset(\tau)}\Bfr \frac{2}{\alpha^2}\E_\emptyset(e^{\alpha\tau})e^{\alpha \E_\emptyset(\tau)} \right) \frac{1}{\alpha^{n-2}}\,.
\]
Applying Bernstein inequality again, we obtain that
\begin{align*}
 &\P_\emptyset\left(
\left|
\sum_{k=1}^{\lfloor T/\E_\emptyset(\tau)\rfloor} (\tau_k-\tau_{k-1}-\E_\emptyset(\tau))
\right|
\ge \frac{T\eps-   |b-a| \E_\emptyset(\tau)}{2|b-a|}
\right)
\\
&\qquad\le
2 \exp\left(-\frac{\Bigl( T\varepsilon-|b-a|\E_\emptyset(\tau) \Bigr)^2}{4 \left(2v + (T\varepsilon-|b-a|\E_\emptyset(\tau))c \right)}  \right).
\end{align*}
Equation \eqref{etape6} gives that
\begin{align*}
\P_\emptyset\biggl( \biggl|\frac1T \int_0^T  f(X_t)\,dt - \pi_A f \biggr| \ge \varepsilon \biggr)
&\le 4 \exp\left(-\frac{\Bigl( T\varepsilon-|b-a|\E_\emptyset(\tau) \Bigr)^2}{4 \left(2v + (T\varepsilon-|b-a|\E_\emptyset(\tau))c \right)}  \right).
\end{align*}
To prove the second part of Corollary \ref{cor:dev_simple} we have to solve
\begin{equation}\label{eq:aux}
\eta=4
 \exp\left(-\frac{\Bigl( T\varepsilon-|b-a|\E_\emptyset(\tau) \Bigr)^2}{4 \left(2v + (T\varepsilon-|b-a|\E_\emptyset(\tau))c \right)}  \right)
\end{equation}by expressing $\varepsilon$ as function of $\eta$, for any $\eta\in (0,1)$.

Let us define the following decreasing bijection from $\R_+$ into $\R_-$:
\[
\varphi(x)=-\frac{x^2}{4(2v+cx)}\,.
\]
The solution of \eqref{eq:aux} is then $\eps_\eta=(|b-a|\E_\emptyset(\tau)+x_0)/T$ where $x_0$ is the unique positive solution of
\[
\varphi(x)=\log\Big(\frac{\eta}{4}\Big)\quad \Leftrightarrow \quad x^2+4c\log\big(\frac{\eta}{4}\big) x + 8v\log\big(\frac{\eta}{4}\big)=0\,.
\]
Computing the roots of this second order polynomial, we can show that there always exist one negative and one positive root as soon as $\eta<4$. More precisely,
\[
x_0=-2c\log\big(\frac{\eta}{4}\big)+\sqrt{4c^2\log^2\big(\frac{\eta}{4}\big)-8 v \log\big(\frac{\eta}{4}\big)}\,,
\]
which concludes the proof.

\appendix

\section{Appendix}
\small
\subsection{Proof of Proposition \ref{prop:couplage}}
\label{app:proof}
Before proving Proposition \ref{prop:couplage}, we start with a lemma showing that Assumption~\eqref{cond_N0} implies a milder condition which will be used repeatedly in the proof of the proposition.

\begin{lem}\label{lem:conditionN0}
Suppose that Assumption~\eqref{cond_N0} is satisfied. Then for any nonnegative random variable $U$ and $r>0$,
\begin{equation*}
\P_{\mathfrak{m}}\bigg(\int_U^{U+r} \int_{(-\infty,0]}h^+(t-s) \,N^0(ds) \,dt < +\infty,~~ U< +\infty \bigg)=\P_{\mathfrak{m}}(U<+\infty)\,.
\end{equation*}
\end{lem}

\begin{proof}
First note that, for every integer $n$,
\[
\int_0^n\int_{(-\infty,0]}h^+(t-s) \,N^0(ds) dt < +\infty\,,\;
 \P_{\mathfrak{m}}-\text{a.s.},
\]
using condition \eqref{cond_N0} and the Fubini-Tonelli theorem. This leads easily to
\[
\P_{\mathfrak{m}}\bigg( \forall n \ge 0,~\int_0^n\int_{(-\infty,0]}h^+(t-s) \,N^0(ds) dt < +\infty \bigg) = 1\,,
\]
and, for a positive real number $r$, to
\[
\P_{\mathfrak{m}}\bigg(\forall u\ge0,~\int_u^{u+r} \int_{(-\infty,0]}h^+(t-s) \,N^0(ds) dt < +\infty\bigg)=1\,,
\]
which gives the announced result.
\end{proof}

\begin{proof} [Proof of Proposition \ref{prop:couplage}]
Proofs of both \ref{prop:couplage-i}) and \ref{prop:couplage-ii}) will be obtained by induction on the successive atoms of $N^h$.

\paragraph{Proof of \ref{prop:couplage-i}): initialization.}
Let
\begin{align}
\label{def:Lambda0}
&\Lambda^h_0(t)=\biggl(\lambda+\int_{(-\infty,0]}h(t-s) \,N^0(ds)\biggr)^+\,,
&& t>0\,,
\\
\label{def:U1h}
&U_1^h=\inf\biggl\{u > 0 : \int_{(0,u]} \int_{(0,\Lambda_0^h(v)]} \,Q(dv,d\theta)>0 \biggr\}\,,
\end{align}
with the usual convention that $\inf \emptyset = +\infty$.
First note that conditionally on $N^0$,
\[
Q(\{(v,\theta) \in (0,\varepsilon] \times (0,+\infty) : \theta \le \Lambda_0^h(v)\})
\]
 follows a Poisson law with parameter $\int_0^{\varepsilon} \Lambda^h_0(t) dt$. Using Assumption \eqref{cond_N0} and Lemma \ref{lem:conditionN0}, we can find $\varepsilon_0 >0$ such that
 $\int_0^{\varepsilon_0} \int_{(-\infty,0]}h^+(t-s) \,N^0(ds)  dt < +\infty$.
 We thus have, $\P_{\mathfrak{m}}$-a.s.,
\begin{align*}
 \int_0^{\varepsilon_0}  \Lambda^h_0(t)  dt
 &= \int_0^{\varepsilon_0} \biggl(\lambda+\int_{(-\infty,0]}h(t-s) \,N^0(ds)\biggr)^+\,  dt
 \\
&\le \lambda \varepsilon_0 +\int_0^{\varepsilon_0} \int_{(-\infty,0]}h^+(t-s) \,N^0(ds)  dt < + \infty\,.
\end{align*}
Consequently, $Q(\{(v,\theta) \in (0,\varepsilon_0] \times (0,+\infty) : \theta \le \Lambda_0^h(v)\})$ is finite
$\P_{\mathfrak{m}}$-a.s.
Hence  $U_1^h>0$ $\P_{\mathfrak{m}}$-a.s. If $U_1^h=+\infty$ then $N^h=N^0$,
and we define $U_k^h =+\infty$ for all $k \ge 2$.
Else,  $U_1^h$ is the first atom on $(0,+\infty)$ of the point process of
 conditional intensity $\Lambda_0^h$. Since $\Lambda^h_0(t)= \Lambda^h(t)$ for $t\in (0,U_1^h]$, thus $U^h_1$ is
 also the first atom of $N^h$ on $(0,+\infty)$.

\paragraph{Proof of \ref{prop:couplage-i}): recursion.}
Assume that  we have built $U_1^h, \dots  ,U_k^h$ such that on the event $\{U_k^h < +\infty\}$
these are the first $k$ atoms of $N^h$ in increasing order. We are going to construct $U_{k+1}^h$,
which will be an atom of $N^h$ greater than $U_k^h$.

On $\{U_k^h = +\infty\}$ we set $U_{k+1}^h = +\infty$.
Henceforth, we work on $\{U_k^h < +\infty\}$. Let
\begin{align}
\label{def_int_part}
&\Lambda^h_k(t) =\biggl(\lambda+\int_{(-\infty,0]}h(t-s) \,N^0(ds) + \int_{(0,U_k^h]} h(t-s) \,N^h (ds)\biggr)^+\,,
\qquad
t>0\,,
\\
&U_{k+1}^h=\inf\biggl\{u > U_k^h : \int_{(U_k^h,u]}\int_{(0,\Lambda_k^h(v)]} \,Q(dv,d\theta)>0 \biggr\}\,.
\notag
\end{align}
As in Step 1, we first prove that there exists $\varepsilon >0$ such that $Q(\mathcal{R}_{\varepsilon})$ is a.s. finite, where
\[
\mathcal{R}_{\varepsilon}=\{(v,\theta) : v \in (U_k^h,U_k^h +\varepsilon],\, \theta \in (0,\Lambda_k^h(v)]\}\,.
\]
Since the random function $\Lambda_k^h$ is measurable with respect to $\mathcal{F}_{U_k^h}$, conditionally on $\mathcal{F}_{U_k^h}$, $Q(\mathcal{R}_{\varepsilon})$ follows a Poisson law with parameter $\int_{U_k^h}^{U_k^h +\varepsilon} \Lambda^h_k(t) dt$
(see Lemma~\ref{lem:Q})
so that
$$\P(Q(\mathcal{R}_{\varepsilon}) < +\infty) = \E\Big(\P(Q(\mathcal{R}_{\varepsilon}) < +\infty \,|\,\mathcal{F}_{U_k^h} )\Big) = \E\left(\P\left(\int_{U_k^h}^{U_k^h +\varepsilon} \Lambda^h_k(t) dt < +\infty \,\bigg|\,\mathcal{F}_{U_k^h} \right)\right).$$
Using the fact that $x \le x^+$ and the monotonicity of $x \mapsto x^+$, we  obtain
from~\eqref{def_int_part} that
\begin{align*}
\int_{U_k^h}^{U_k^h +\varepsilon} \Lambda^h_k(t) dt
\le  \lambda \varepsilon
&+ \int_{U_k^h}^{U_k^h +\varepsilon} \int_{(-\infty,0]}h^+(t-s) \,N^0(ds) dt
\\
&+ \int_{U_k^h}^{U_k^h +\varepsilon} \int_{(0,U_k^h]} h^+(t-s) \,N^h (ds) dt\,.
\end{align*}
On $\{U_k^h <+\infty\}$ the second term in the r.h.s.\ is finite thanks to Assumption~\eqref{cond_N0} and Lemma \ref{lem:conditionN0}. It is thus also finite, a.s., on $\{U_k^h <+\infty\}$, conditionally on $\mathcal{F}_{U_k^h}$.
Now, using the Fubini-Tonelli Theorem and Assumption \eqref{cond_h}, we obtain that
\begin{align*}
 \int_{U_k^h}^{U_k^h +\varepsilon} \int_{(0,U_k^h]} h^+(t-s) \,N^h (ds) dt &= \int_{(0,U_k^h]} \biggl(\int_{U_k^h}^{U_k^h +\varepsilon} h^+(t-s)dt \biggr) \,N^h (ds) \\
  & \le \|h^+\|_1 \,N^h((0,U_k^h]) = k \|h^+\|_1 < +\infty.
\end{align*}
This concludes the proof of the finiteness of $\int_{U_k^h}^{U_k^h +\varepsilon} \Lambda^h_k(t) dt $, so that $Q(\mathcal{R}_{\varepsilon}) <+\infty$, $\P_{\mathfrak{m}}$-a.s.

If  $Q(\mathcal{R}_{\varepsilon}) $ is null then $U_{k+1}^h = +\infty$ and thus $N^h = N^0 +\sum_{i=1}^k \delta_{U_i^h}$. Else, $U_{k+1}^h$ is actually a minimum, implying that $U_k^h < U_{k+1}^h$ and, since $\Lambda^h$ and $\Lambda^h_k$ coincide on $(0,U_{k+1}^h)$, that $U_{k+1}^h$ is the $(k+1)$-th atom of $N^h$.

We have eventually proved by induction the existence of a random  sequence $(U_k^h)_{k \ge 1}$ which
strictly increases up to the first rank it possibly hits $+\infty$ and stays there.
On the event that this first rank is finite,
the finite $U_k^h$ are exactly the atoms of the random point process $N^h$ on $(0, +\infty)$.

To complete the proof, it is thus enough to prove that $\lim_{k \rightarrow +\infty} U_k^h = +\infty$, $\P_{\mathfrak{m}}$-a.s. For this, we compute $\E_{\mathfrak{m}}(N^h(0,t))$ for $t >0$. For all $k \ge 1$,
\begin{align*}
\E_{\mathfrak{m}}\big(N^h (0,t  \wedge U_k^h)\big)&=\E_{\mathfrak{m}}\bigg(\int_0^{t\wedge U_k^h} \Lambda^h(u)du\bigg)\\
&=\E_{\mathfrak{m}}\bigg(\int_0^{t\wedge U_k^h} \bigg(\lambda+ \int_{(-\infty ,u)} h(u-s)\,N^h(ds)\bigg)^+\ du\bigg)\\
& \le \lambda t+\E_{\mathfrak{m}}\bigg(\int_0^t\int_{(-\infty,0]} h^+(u-s)\,N^0(ds)du \bigg)
\\
&\hphantom{\null  \le \lambda t}
+ \E_{\mathfrak{m}}\bigg(\int_0^{t\wedge U_k^h}\int_{(0,u)}h^+(u-s)\,N^h(ds) du \bigg).
\end{align*}
Using the nonnegativity of $h^+$ and Assumption \eqref{cond_N0},
\[
 \E_{\mathfrak{m}}\bigg(\int_0^t \int_{(-\infty,0]} h^+(u-s)\,N^0(ds)du \bigg)
 \le \int_0^t \E_{\mathfrak{m}}\bigg( \int_{(-\infty,0]} h^+(u-s)\,N^0(ds) \bigg) du <+\infty\,.
\]
For the last term, we use again the  Fubini-Tonelli theorem and obtain
\begin{align*}
\E_{\mathfrak{m}}\bigg(\int_0^{t\wedge U_k^h} \int_{(0,u)} h^+(u-s) \,N^h(ds)\ du\bigg)& = \E_{\mathfrak{m}}\bigg(\int_{(0,t\wedge U_k^h)} \int_s^{t\wedge U_k^h} h^+(u-s)du \,N^h(ds)\bigg)\\
& \leq \|h^+\|_1 \,\E_{\mathfrak{m}}\bigg( N^h(0,t\wedge U_k^h)\bigg).
\end{align*}
These three  inequalities and the fact that $\|h^+\|_1<1$, see Assumption \eqref{cond_h}, yield that
\begin{align}
0\leq  \E_{\mathfrak{m}}\big(N^h(0,t\wedge U_k^h)\big) \leq \frac{1}{1-\|h^+\|_1} \bigg(\lambda t + \int_0^t \E_{\mathfrak{m}}\bigg( \int_{(-\infty,0]} h^+(u-s)\,N^0(ds) \bigg) du\bigg) \label{etape5}
\end{align}
where the upper bound is finite and independent of $k$.

As a consequence, we necessarily have that $\lim_{k\rightarrow +\infty} U_k^h=+\infty$ a.s.,
which we now prove by contradiction. If $\P(\lim_{k\rightarrow +\infty} U_k^h<+\infty)>0$ then
there would exist $T>0$ and $\Omega_0$ such that $\P(\Omega_0)>0$ and $\lim_{k\rightarrow +\infty} U_k^h\leq T$ on $\Omega_0$. But this would entail that $\E_{\mathfrak{m}}(N^h(0,T\wedge U_k^h))\geq (k-1) \P_{\mathfrak{m}}(\Omega_0)$ which converges to $+\infty$ with $k$ and cannot be upper bounded by \eqref{etape5}. \\
Note additionally that once we know that $\lim_{k\rightarrow +\infty} U_k^h=+\infty$, a.s., we can use
the Beppo-Levi theorem, which leads to $\E_{\mathfrak{m}}\big(N^h(0,t)\big) < +\infty$ for all $t>0$.

Note that uniqueness comes from the algorithmic construction of the sequence $(U^h_k)_{k\geq 1}$.

\paragraph{Proof of \ref{prop:couplage-ii}).}
The assumptions of the theorem are valid both for $h$ and for $h^+$, and the result \ref{prop:couplage-i})
which we have just proved allows to construct strong solutions $N^h$ and $N^{h^+}$ of Eq.~\eqref{eq:Ng} driven by the same Poisson point process $Q$. Proving \ref{prop:couplage-ii}) is equivalent to showing that the atoms of $N^h$ are also atoms of $N^{h^+}$, which we do using the following recursion.

If $U_1^h = +\infty$ then $N^h$ as no atom on $(0,+\infty)$ and there is nothing to prove.

Else, we first show that the first atom $U^h_1$ of $N^h$ is also an atom of $N^{h^+}$. The key point is to establish that
\begin{equation} \label{eq:comp_int}
\forall t \in (0,U_1^h),~~ \Lambda^h (t)\leq \Lambda^{h^+}(t).
\end{equation}
Indeed, from the definition of $U_1^h$, there exists an atom of the Poisson measure $Q$ at some $(U_1^h,\theta)$ with $\theta \le \Lambda^h\big((U_1^h)_-\big)$. If (\ref{eq:comp_int}) is true we may deduce that $(U_1^h,\theta)$ is also an atom of $Q$ satisfying $\theta \le \Lambda^{h^+}\big((U_1^h)_-\big)$, and thus that $U_1^h$ is also an atom of $N^{h^+}$.

We now turn to the proof of (\ref{eq:comp_int}).
For every $t\in (0,U_1^h)$, we clearly have
\[
\Lambda^h (t) = \Lambda^h_0(t)
\triangleq \bigg(\lambda+\int_{(-\infty,0]}h(t-s) \,N^0(ds)\bigg)^+\,,
\]
we use the fact that $x \mapsto x^+$ is nondecreasing on $\R$ to obtain that
\[
\Lambda^h (t)  \le \lambda+\int_{(-\infty,t)}h^+(t-s) \,N^{h^+}(ds) \triangleq \Lambda^{h^+}(t)\,.
\]
We now prove that if $U_1^h, \dots , U_k^h$ are atoms of $N^{h^+}$ and $U_{k+1}^h<+\infty$ then $U_{k+1}^h$ is also an atom of $N^{h^+}$.
By construction, $\Lambda^h(t)=\Lambda^h_k(t)$ for all $t \in (0, U_{k+1}^h)$, and there exists $\theta >0$ such that $(U_{k+1}^h, \theta)$ is an atom of $Q$ satisfying $\theta \le \Lambda^h((U_{k+1}^h)_-)$. To obtain that $U_{k+1}^h$ is also an atom of $N^{h^+}$, it is thus enough to prove that
\begin{equation*}
 \forall t \in [U_k^h,U_{k+1}^h),~~ \Lambda^h (t)\leq \Lambda^{h^+}(t).
\end{equation*}
Using that $h \le h^+$ and the induction hypothesis that the first $k$ atoms $U_1^h, \dots , U_k^h$
of $N^h$ are also atoms of $N^{h^+}$, we obtain for all $t \in (U_k^h,U_{k+1}^h)$ that
\[
\int_{(0,U_k^h]} h(t-s) \,N^h (ds) \le \int_{(0,U_k^h]} h^+(t-s) \,N^h (ds) \le \int_{(0,t)} h^+(t-s) \,N^{h^+} (ds)\,.
\]
This upper bound and the definition \eqref{def_int_part} of $\Lambda_k^h$ yield that, for all $t \in (U_k^h,U_{k+1}^h)$,
\[
\Lambda^h_k(t) \le \Lambda^{h^+}(t)\,,
\]
and since $\Lambda_k^h$ and $\Lambda^h$ coincide on $(0, U_{k+1}^h)$, we have finally proved that $U_{k+1}^h$ is an atom of $N^{h^+}$. This concludes the proof of the proposition.
\end{proof}

\subsection{Extension to the more general setting of Remark~\ref{rmk-phi-coupling}}
\label{app:extension}

As noted in Remark \ref{rmk-phi-coupling}, the results of this article can be extended to a more general setting.
A critical point for this extension is to construct a coupling of the Hawkes process $N^{h,\phi}$ with a Hawkes process $N^g$ satisfying Definition \ref{def:Hawkes} for a nonnegative function $g$,
in such a way that $N^{h,\phi} \le N^g$ (thinning).
Then $N^g=\emptyset$ implies that $N^{h,\phi} =\emptyset$ and in particular this allows to derive exponential bounds on the renewal time $\tau$ of  $N^{h,\phi}$.
\begin{prop}
Assume that $N^{h,\phi}$ is a Hawkes process with conditional intensity $\Lambda^{h,\phi}$ defined in
\eqref{def:lambda_gen} and that the function $\phi$ and $h$ satisfy that there exist $\lambda$ and $a$ in $[0,\infty)$ such that $\forall x\in\R$
\[
\phi(x)\le \lambda + a x^+\,\quad \text{and}\quad\,
a\int h^+ <1.
\]
Let us denote $g=a h^+$.
Then there exists a coupling of $N^{h,\phi}$ with a Hawkes process $N^g$ in the sense of Definition \ref{def:Hawkes} such that a.s. $N^{h,\phi}\le N^{g}$.
\end{prop}
\begin{proof}[Scheme of the proof]
Similarly to the previous case, a key point is to upper-bound the intensity $\Lambda^{h,\phi}$ on given time intervals
\begin{align*}
\phi\left(\int_{(-\infty,t)} h(t-u) N^{h,\phi} (du)\right)
&\le \lambda + a \left(\int_{(-\infty,t)} h(t-u) N^{h,\phi} (du)\right)^+
&&\text{(by assumption)}
\\
&\le \lambda +  \int_{(-\infty,t)} g(t-u) N^{h,\phi} (du)
&&\text{(since $g = a h^+$)}
\\
&\le \lambda + \int_{(-\infty,t)} g(t-u) N^g (du)
&&\text{(thinning)}
\end{align*}
and thus it is possible at each point $U$ of $N^{g}$ to either include it into $N^{h,\phi}$ with probability
\[
\frac{\phi\left(\int_{(-\infty,U)} h(U-u) N^{h,\phi} (du)\right)}{\lambda +  \int_{(-\infty,U)} g(U-u) N^{g} (du)} \le1
\]
or else to reject it, independently of the rest. Then the conditional intensity of $N^{h,\phi}$ is given by
$\frac{\phi\left(\int_{(-\infty,U)} h(t-u) N^{h,\phi} (du)\right)}{\lambda +  \int_{(-\infty,t)} g(t-u) N^{g} (du)}
\left(\lambda +  \int_{(-\infty,t)} g(t-u) N^{g} (du)\right)
= \phi\left(\int_{(-\infty,t)} h(t-u) N^{h,\phi} (du)\right)$.
\end{proof}

\subsection{Return time for $M/G/\infty$ queues}

We now state a general result for the tail behavior of the time of return to zero $\mathcal{T}_1$  of a $M/G/\infty$
queue with a service time admitting exponential moments. All queues  in this section start empty.

We recall that a $M/G/\infty$ queue has a Poisson process of customer arrivals with i.i.d.~service times
with a general distribution, and each customer starts its service immediately at arrival and leaves the system
at its completion. For the Hawkes process with nonnegative reproduction function,
we  consider the ancestors
to be customers (arriving as a Poisson process of intensity $\lambda$) with service times distributed as
$\widetilde{H}_1(A) \triangleq H_1+A$
where $H_1$ is a cluster length (see Section~\ref{subsec:cluster}), and then
the queue empties exactly at the hitting times of $\emptyset$
by the auxiliary Markov process.

This result has an interest in itself, independently of the Hawkes process
interpretation.
Its proof  is based on the computation of the Laplace transform $\E(\mathrm{e}^{-s\mathcal{T}_1})$ on the
half-plane $\{s\in\C : \Re(s) > 0\}$ by Tak\'acs \cite{Takacs1956,Takacs1962}.
We extend analytically this Laplace transform to $\{s\in\C : \Re(s) > s_c\}$
for an appropriate $s_c<0$, which yields exponential moments.

\begin{thm}
\label{thm:carl}
Consider a $M/G/\infty$ queue with arrival rate $\lambda >0$ and generic service duration $H$ satisfying
for some $\gamma >0$ that, for $t\ge0$,
\[
\P(H>t) \triangleq1-G(t) = O(\mathrm{e}^{-\gamma t})\,.
\]
Let $V_1$ denote the arrival time of the first customer,
$\mathcal{T}_1$ the subsequent time of return of the queue to zero,
and $B=\mathcal{T}_1-V_1$ the corresponding busy period.
\begin{enumerate}[\rm a)]
\item
\label{exp-moment-B}
If $\beta <\gamma$ then $\E(\mathrm{e}^{\beta B}) < \infty$. In particular $\P(B \ge t) = O(\mathrm{e}^{-\beta t})$.
\item
\label{exp-moment-tau1}
If $\lambda <\gamma$, then $\P(\mathcal{T}_1 \ge t) =  O(\mathrm{e}^{-\lambda t})$.
If $\gamma \le \lambda$, for $\alpha<\gamma$ then $\P(\mathcal{T}_1 \ge t) =  O(\mathrm{e}^{-\alpha t})$.
\end{enumerate}
\end{thm}
\begin{proof}
We have $\mathcal{T}_1=V_1+B$, and
the strong Markov property of the Poisson process yields that $V_1$ and $B$ are independent.
Since $V_1$ is exponential of parameter $\lambda$,
we need mainly to study $B$.
Tak\'acs has proved in~\cite[Eq.~(37)]{Takacs1956}, see also \cite[Theorem~1~p.~210]{Takacs1962},
that the Laplace transform of $\mathcal{T}_1$ satisfies
\begin{equation}
\label{A-laplace}
\E(\mathrm{e}^{-s \mathcal{T}_1}) =
 1-
\frac1{\lambda + s} \frac1{\int_0^\infty \mathrm{e}^{-st - \lambda \int_0^t[1-G(u)]\,\mathrm{d} u}\,\mathrm{d}t}\,,
\qquad
s\in\C\,,\;\Re(s) > 0\,.
\end{equation}
 Since the Laplace transform of $V_1$ is $\frac{\lambda}{\lambda + s}$,
 the Laplace transform of $B$ satisfies
\begin{equation}
\label{B-laplace}
\E(\mathrm{e}^{-s B}) =
\frac{\lambda + s}{\lambda} -
\frac1{\lambda} \frac1{\int_0^\infty \mathrm{e}^{-st - \lambda \int_0^t[1-G(u)]\,\mathrm{d} u}\,\mathrm{d}t}\,,
\qquad
s\in\C\,,\;\Re(s) > 0\,.
\end{equation}

There is an apparent singularity in the r.h.s.\ of \eqref{A-laplace} and of \eqref{B-laplace}, since the integral term increases to infinity as $s$ decreases to $0$.  This  is normal, since these formulae remain valid for heavy tailed service.
Moreover, \eqref{A-laplace}  is proved in~\cite{Takacs1956,Takacs1962}
using
the Laplace transform of a measure with infinite mass.
We shall remove this apparent singularity and compute the abscissa of convergence of the Laplace transform
in the l.h.s.\ of~\eqref{B-laplace} .

The main point to prove is that the abscissa of convergence $\sigma_c$ of the Laplace transform in the l.h.s.\ of~\eqref{B-laplace}
satisfies $\sigma_c \le -\gamma$.
In order to remove the apparent singularity in the r.h.s.\ of \eqref{B-laplace}, we use integration by parts:
on the half-line $\{s\in\R : s> 0\}$,
\begin{align}
\label{int-parts}
\int_0^\infty \mathrm{e}^{-st - \lambda \int_0^t[1-G(u)]\,\mathrm{d} u}\,\mathrm{d}t
&=  \left[\frac{\mathrm{e}^{-st}}{-s} \mathrm{e}^{- \lambda \int_0^t[1-G(u)]\,\mathrm{d} u} \right]_{t=0}^\infty
\notag\\
&\qquad
- \int_0^\infty \frac{\mathrm{e}^{-st}}{-s} (-\lambda[1-G(t)])\, \mathrm{e}^{- \lambda \int_0^t[1-G(u)]\,\mathrm{d} u}\,\mathrm{d}t
\notag\\
&= \frac1s - \frac{\lambda}s \int_0^\infty [1-G(t)]\, \mathrm{e}^{-st- \lambda \int_0^t[1-G(u)]\,\mathrm{d} u}\,\mathrm{d}t\,.
\end{align}
After inspection of the integral on the r.h.s., since  $1-G(t) = O(\mathrm{e}^{-\gamma t})$ and
\be
\lambda  \int_0^\infty [1-G(t)]\, \mathrm{e}^{- \lambda \int_0^t[1-G(u)]\,\mathrm{d} u}\,\mathrm{d}t
= \left[-\mathrm{e}^{- \lambda \int_0^t[1-G(u)]\,\mathrm{d} u} \right]_{t=0}^\infty = 1 - \mathrm{e}^{- \lambda \E(H)} < 1,
\ee
we are able to define a constant $\theta <0$ and an analytic function $f$ by setting
\begin{align}
\theta &= \inf \left\{
s \le 0 : \lambda  \int_0^\infty [1-G(t)] \,\mathrm{e}^{-st- \lambda \int_0^t[1-G(u)]\,\mathrm{d} u}\,\mathrm{d}t <1
\right\} \vee(- \gamma) \,,
\notag\\
\label{prolong}
f(s) &=
\frac{\lambda + s}{\lambda} - \frac{s}{\lambda}
\frac1{1 - \lambda \int_0^\infty [1-G(t)]\, \mathrm{e}^{-st- \lambda \int_0^t[1-G(u)]\,\mathrm{d} u}\,\mathrm{d}t}\,,
\qquad
s\in\C\,,\;
 \Re(s) > \theta \,.
\end{align}

The Laplace transform in the l.h.s.\ of \eqref{B-laplace} has an abscissa of convergence $\sigma_c\le 0$ and
is  analytic in the half-plane $\{s\in\C : \Re(s) > \sigma_c\}$, see Widder~\cite[Theorem~5a~p.57]{Widder1941}.
Both this Laplace transform and $f$ are analytic in the domain  $\{s\in\C : \Re(s) > \max(\theta, \sigma_c)\}$,
and since these two analytic functions coincide there on the half-line $\{s\in\R : s> 0\}$ they must coincide in the whole domain,
see Rudin~\cite[Theorem~10.18~p.208]{rudin}, so that
\[
\E(\mathrm{e}^{-s B})
=f(s) \,,
\qquad
s\in\C \,,\;
\Re(s) > \max(\theta, \sigma_c)\,.
\]
This Laplace transform must  have an analytic singularity at $s=\sigma_c$,
see Widder~\cite[Theorem~5b~p.~58]{Widder1941},
and since $f$ is analytic in $\{s\in\C : \Re(s) > \theta\}$
necessarily $\sigma_c \le \theta$.

Since $\theta<0$, by monotone convergence
\[
\lim_{s\to\theta^+}f(s) =
\frac{\lambda + \theta}{\lambda} - \frac{\theta}{\lambda}
\frac1{1 - \lambda \int_0^\infty [1-G(t)]\, \mathrm{e}^{-\theta t- \lambda \int_0^t[1-G(u)]\,\mathrm{d} u}\,\mathrm{d}t}
= \E(\mathrm{e}^{- \theta B}) \in [1,\infty]\,,
\]
which implies that
 $
 \lambda \int_0^\infty [1-G(t)]\, \mathrm{e}^{-\theta t- \lambda \int_0^t[1-G(u)]\,\mathrm{d} u}\,\mathrm{d}t <1,
 $
and thus that $\theta = -\gamma$.

We conclude that  $\sigma_c \le -\gamma$.
 Thus, if $\beta <\gamma$ then $\E(\mathrm{e}^{\beta B}) < \infty$,
and   $\P(B \ge t) = O(\mathrm{e}^{-\beta t})$  using the Markov inequality.
Moreover, if $\P(B \ge t) = O(\mathrm{e}^{-\alpha t})$ then
\begin{align*}
\P(\mathcal{T}_1 \ge t) = \P(B+V_1 \ge t) &= \mathrm{e}^{-\lambda t}
+ \lambda \int_0^t \mathrm{e}^{-\lambda u}\P(B \ge t-u)\,\mathrm{d}u
\\
&\le \mathrm{e}^{-\lambda t} + C \int_0^t \mathrm{e}^{-\lambda u -\alpha (t-u)} \,\mathrm{d}u\,,
\end{align*}
hence if $\lambda <\gamma$ then choosing  $\lambda < \alpha < \gamma$ yields that
\[
\P(\mathcal{T}_1 \ge t)  \le  \mathrm{e}^{-\lambda t}
+C \mathrm{e}^{-\lambda t}\int_0^t \mathrm{e}^{ -(\alpha-\lambda) (t-u)} \,\mathrm{d}u
\le [1 + C/(\alpha-\lambda)]\mathrm{e}^{-\lambda t} ,
\]
and if  $\alpha <\gamma \le \lambda $ then
\[
\P(\mathcal{T}_1 \ge t)  \le \mathrm{e}^{-\lambda t} + C\mathrm{e}^{-\alpha t}\int_0^t \mathrm{e}^{-(\lambda -\alpha)u} \,\mathrm{d}u
\le  \Bigl[1 + \frac{C}{\lambda-\alpha}\Bigr]\mathrm{e}^{-\alpha t} \,.
\qedhere
\]
\end{proof}

We now provide a corollary to the previous result.

\begin{prop}\label{prop:appendice-suitecarl}
Consider a $M/G/\infty$ queue with arrival rate $\lambda >0$ and generic service duration $H$ satisfying
for some $\gamma >0$ that
\[
\P(H>t) = O(\mathrm{e}^{-\gamma t})\,.
\]
Let $Y_t$ denote the number of customers at time $t\ge0$,
and for each $E\ge 0$ let
\begin{equation}
\tau_E=\inf\{ t\geq E : Y_t=0\}
\end{equation}
be the first hitting time of zero after $E$.
If $\lambda<\gamma$ then let $\alpha=\lambda$, and if $\gamma\leq \lambda$ then let $0<\alpha<\gamma$.
Then there exists a constant $C<\infty$ such that
\[
\P(\tau_E \ge t)\leq \lambda C E\ e^{-\alpha(t-E)}\,,
\quad
\forall t\geq E\,.
\]
\end{prop}
\begin{proof}
The successive return times to zero $(\mathcal{T}_k)_{k\geq 0}$ of the process $(Y_t)_{t\geq 0}$  have been defined in \eqref{def:tcalk}. The events $\{\mathcal{T}_{k-1}\leq E,\ \mathcal{T}_k>E\}$ for $k\geq 1$ define a partition of $\Omega$ and, for $t>E$,
\begin{align*}
\P(\tau_E \ge t)& = \sum_{k=1}^{+\infty} \P\big(\tau_E \ge t,\,\mathcal{T}_{k-1}\leq E,\,\mathcal{T}_k>E\big)
\\
& = \sum_{k=1}^{+\infty}\P\big( \mathcal{T}_{k-1}\leq E,\,\mathcal{T}_k\ge t\big)
\\
& = \sum_{k=1}^{+\infty}\E\Big( \ind_{\{\mathcal{T}_{k-1}\leq E\}} \P\big(\mathcal{T}_k\ge t \,|\,\mathcal{F}_{\mathcal{T}_{k-1}}\big)\Big)\\
&\leq   \sum_{k=1}^{+\infty}\E\Big( \ind_{\{\mathcal{T}_{k-1}\leq E\}} \P\big(\mathcal{T}_k-\mathcal{T}_{k-1} \ge t-E \,|\,\mathcal{F}_{\mathcal{T}_{k-1}}\big)\Big)
\end{align*}
so that, since $\mathcal{T}_k-\mathcal{T}_{k-1} $ is independent of $\mathcal{F}_{\mathcal{T}_{k-1}}$ and distributed as $\mathcal{T}_1$,
\[
\P(\tau_E \ge t)
\le \sum_{k=1}^{+\infty}\E\Big( \ind_{\{\mathcal{T}_{k-1}\leq E\}} \Big) \P\big(\mathcal{T}_1 \ge t-E\big)
=  \P\big(\mathcal{T}_1 \ge t-E\big) \E\Bigg(\sum_{k=1}^{+\infty} \ind_{\{\mathcal{T}_{k-1}\leq E\}} \Bigg)\,.
\]
 By Theorem~\ref{thm:carl}, under the assumptions there exists a constant $C$ such that
\[
\P\big(\mathcal{T}_1 \ge t-E\big)\leq C \mathrm{e}^{\alpha(t-E)}\,.
\]
Moreover
$
\sum_{k=1}^{+\infty} \ind_{\{\mathcal{T}_{k-1}\leq E\}}
$
is  the number of returns to zero before time $E$. It is
 bounded by the number of arrivals between times~$0$ and $E$, which follows a Poisson law of parameter
 and expectation $\lambda E$. This leads to the announced inequality.
\end{proof}

\subsection{Strong Markov property for homogeneous Poisson point process}

In this appendix, we prove a strong Markov property for homogeneous Poisson point processes on the line.
This classic result is stated in \cite[Proposition~1.18~p.18]{robert_philippe} when the filtration is the
canonical filtration generated by the Poisson point process. Here, the filtration $(\mathcal{F}_t)_{t\ge 0}$
may contain additional information, for example coming from configurations on $\R_-$.

\begin{lem}
\label{lem:Q}
Let $Q$ be a $(\mathcal{F}_t)_{t\ge0}$-Poisson point process on $(0,+\infty)\times (0,+\infty)$ with
unit intensity.
Then $Q$ is a strong $(\mathcal{F}_t)_{t\ge0}$-Markov process in the following sense: for any
stopping time $T$ for $(\mathcal{F}_t)_{t\ge0}$,
conditionally on $T<\infty$  the shifted process $S_TQ$ defined by \eqref{shift-Q}
is a $(\mathcal{F}_{T+t})_{t\ge0}$-Poisson point process with unit intensity.
\end{lem}

\begin{proof}
It is enough to prove that, for any stopping time $T$ and $h,a>0$,
conditionally on $T<\infty$
the random variable
$Q((T,T+h]\times (0,a])$ is $\mathcal{F}_{T+h}$-measurable, independent of $\mathcal{F}_T$, and
Poisson of parameter $h a$. Indeed, in order to prove the strong Markov property at a given
stopping time $T$, it is enough to apply the above to the stopping times $T+t$ for $t>0$
in order to see that $S_TQ$ satisfies {that for every $t,h,a>0$, the random variable
$Q((t,t+h]\times (0,a])$ is $\mathcal{F}_{t+h}$-measurable, independent of $\mathcal{F}_t$, and
Poisson of parameter $h a$.}

We first prove this for an arbitrary stopping time $T$ with finite values
belonging to an increasing deterministic sequence $(t_n)_{n\ge1}$.
For each $B$ in $\mathcal{F}_T$ and $k\ge0$,
\begin{align*}
&\P(B\cap\{T<\infty\}\cap \{Q((T,T+h]\times (0,a])=k\})
\\
&\quad =\sum_{n\ge1} \P(B\cap\{T=t_n\}\cap\{Q((t_n,t_n+h]\times (0,a])=k\})
\end{align*}
in which,  by definition of $\mathcal{F}_T$ and since $\mathcal{F}_{t_{n-1}}\subset \mathcal{F}_{t_n}$,
\[
B\cap\{T=t_n\} = (B\cap \{T\le t_n\}) - (B\cap \{T\le t_{n-1}\}) \in \mathcal{F}_{t_n}\,.
\]
The $(\mathcal{F}_t)_{t\ge0}$-Poisson point process property then yields that
\[
\P(B\cap\{T=t_n\}\cap\{Q((t_n,t_n+h]\times (0,a])=k\}) = \P(B\cap\{T=t_n\})\,\mathrm{e}^{-ha}\frac{(ha)^k}{k!}
\]
and summation of the series that
\[
\P(B\cap\{T<\infty\}\cap \{Q((T,T+h]\times (0,a])=k\}) =\P(B\cap\{T<\infty\})\,\mathrm{e}^{-ha}\frac{(ha)^k}{k!}\,.
\]
Hence $Q((T,T+h]\times (0,a])$ is independent of $\mathcal{F}_T$ and
Poisson of parameter $h a$.
Moreover, for $k\ge0$, similarly
\begin{align*}
 &\{T<\infty,\, Q((T,T+h]\times (0,a])=k\} \cap \{T+h\le t\}
 \\
 &\quad= \bigcup_{n\ge1} \{T=t_n,\, Q((t_n,t_n+h]\times (0,a])=k\} \cap \{t_n+h\le t\} \subset\mathcal{F}_t
\end{align*}
and hence $Q((T,T+h]\times (0,a])$ is $\mathcal{F}_{T+h}$-measurable.

In order to extend this to a general stopping time $T$, we approximate it by above by the discrete stopping times
\[
T_n = \sum_{k=1}^{+\infty} \frac{k}{2^n} \ind_{\{\frac{k-1}{2^n} < T \le \frac{k}{2^n}\}}\,,
\qquad
n\ge1\,.
\]
Letting $n$ go to infinity, the right continuity of $t\mapsto Q((0,t]\times(0,a])$ and of $(\mathcal{F}_t)_{t\ge0}$
allows to conclude.
\end{proof}

{\footnotesize
\providecommand{\noopsort}[1]{}\providecommand{\noopsort}[1]{}\providecommand{\noopsort}[1]{}\providecommand{\noopsort}[1]{}

}
\end{document}